\documentclass[11pt]{article}

\usepackage{enumerate,xspace}
\usepackage{amsmath,amssymb,wasysym}
\usepackage[all]{xy}
\usepackage{proof}
\usepackage[svgnames]{xcolor}
\usepackage{tikz}

\usepackage{stmaryrd} %need for special interpretation bracket
\usepackage{mathtools}
\usepackage{latexsym}
\usepackage{dsfont}
\usepackage{multicol}
\usepackage{cmll}
\usepackage{multirow}
\usepackage{longtable}

\usepackage{lscape}
\usepackage{array}

\usepackage{hyperref} %hyperlink package
\hypersetup{
    colorlinks,
    citecolor=red,
    filecolor=red,
    linkcolor=blue,
    urlcolor=red
}

\newtheorem{observation}{Remark}[section]
\newtheorem{lemma}[observation]{Lemma}  %%share counter with remark
\newtheorem{theorem}[observation]{Theorem}
\newtheorem{definition}[observation]{Definition}
\newtheorem{example}[observation]{Example}

\newtheorem{proposition}[observation]{Proposition} 
\newtheorem{corollary}[observation]{Corollary}

\makeatletter

\newdimen\w@dth

\def\setw@dth#1#2{\setbox\z@\hbox{\scriptsize $#1$}\w@dth=\wd\z@
\setbox\@ne\hbox{\scriptsize $#2$}\ifnum\w@dth<\wd\@ne \w@dth=\wd\@ne \fi
\advance\w@dth by 1.2em}

\def\t@^#1_#2{\allowbreak\def\n@one{#1}\def\n@two{#2}\mathrel
{\setw@dth{#1}{#2}
\mathop{\hbox to \w@dth{\rightarrowfill}}\limits
\ifx\n@one\empty\else ^{\box\z@}\fi
\ifx\n@two\empty\else _{\box\@ne}\fi}}
\def\t@@^#1{\@ifnextchar_ {\t@^{#1}}{\t@^{#1}_{}}}

\def\t@left^#1_#2{\def\n@one{#1}\def\n@two{#2}\mathrel{\setw@dth{#1}{#2}
\mathop{\hbox to \w@dth{\leftarrowfill}}\limits
\ifx\n@one\empty\else ^{\box\z@}\fi
\ifx\n@two\empty\else _{\box\@ne}\fi}}
\def\t@@left^#1{\@ifnextchar_ {\t@left^{#1}}{\t@left^{#1}_{}}}

\def\two@^#1_#2{\def\n@one{#1}\def\n@two{#2}\mathrel{\setw@dth{#1}{#2}
\mathop{\vcenter{\hbox to \w@dth{\rightarrowfill}\kern-1.7ex
                 \hbox to \w@dth{\rightarrowfill}}%
       }\limits
\ifx\n@one\empty\else ^{\box\z@}\fi
\ifx\n@two\empty\else _{\box\@ne}\fi}}
\def\tw@@^#1{\@ifnextchar_ {\two@^{#1}}{\two@^{#1}_{}}}

\def\tofr@^#1_#2{\def\n@one{#1}\def\n@two{#2}\mathrel{\setw@dth{#1}{#2}
\mathop{\vcenter{\hbox to \w@dth{\rightarrowfill}\kern-1.7ex
                 \hbox to \w@dth{\leftarrowfill}}%
       }\limits
\ifx\n@one\empty\else ^{\box\z@}\fi
\ifx\n@two\empty\else _{\box\@ne}\fi}}
\def\t@fr@^#1{\@ifnextchar_ {\tofr@^{#1}}{\tofr@^{#1}_{}}}

\newdimen\W@dth
\def\setW@dth#1#2{\setbox\z@\hbox{$#1$}\W@dth=\wd\z@
\setbox\@ne\hbox{$#2$}\ifnum\W@dth<\wd\@ne \W@dth=\wd\@ne \fi
\advance\W@dth by 1.2em}

\def\T@^#1_#2{\allowbreak\def\N@one{#1}\def\N@two{#2}\mathrel
{\setW@dth{#1}{#2}
\mathop{\hbox to \W@dth{\rightarrowfill}}\limits
\ifx\N@one\empty\else ^{\box\z@}\fi
\ifx\N@two\empty\else _{\box\@ne}\fi}}
\def\T@@^#1{\@ifnextchar_ {\T@^{#1}}{\T@^{#1}_{}}}

\def\T@left^#1_#2{\def\N@one{#1}\def\N@two{#2}\mathrel{\setW@dth{#1}{#2}
\mathop{\hbox to \W@dth{\leftarrowfill}}\limits
\ifx\N@one\empty\else ^{\box\z@}\fi
\ifx\N@two\empty\else _{\box\@ne}\fi}}
\def\T@@left^#1{\@ifnextchar_ {\T@left^{#1}}{\T@left^{#1}_{}}}

\def\Tofr@^#1_#2{\def\N@one{#1}\def\N@two{#2}\mathrel{\setW@dth{#1}{#2}
\mathop{\vcenter{\hbox to \W@dth{\rightarrowfill}\kern-1.7ex
                 \hbox to \W@dth{\leftarrowfill}}%
       }\limits
\ifx\N@one\empty\else ^{\box\z@}\fi
\ifx\N@two\empty\else _{\box\@ne}\fi}}
\def\T@fr@^#1{\@ifnextchar_ {\Tofr@^{#1}}{\Tofr@^{#1}_{}}}

\def\Two@^#1_#2{\def\N@one{#1}\def\N@two{#2}\mathrel{\setW@dth{#1}{#2}
\mathop{\vcenter{\hbox to \W@dth{\rightarrowfill}\kern-1.7ex
                 \hbox to \W@dth{\rightarrowfill}}%
       }\limits
\ifx\N@one\empty\else ^{\box\z@}\fi
\ifx\N@two\empty\else _{\box\@ne}\fi}}
\def\Tw@@^#1{\@ifnextchar_ {\Two@^{#1}}{\Two@^{#1}_{}}}

\def\to{\@ifnextchar^ {\t@@}{\t@@^{}}}
\def\from{\@ifnextchar^ {\t@@left}{\t@@left^{}}}
\def\tofro{\@ifnextchar^ {\t@fr@}{\t@fr@^{}}}
\def\To{\@ifnextchar^ {\T@@}{\T@@^{}}}
\def\From{\@ifnextchar^ {\T@@left}{\T@@left^{}}}
\def\Two{\@ifnextchar^ {\Tw@@}{\Tw@@^{}}}
\def\Tofro{\@ifnextchar^ {\T@fr@}{\T@fr@^{}}}

\makeatother

\title{Differential Algebras in Codifferential Categories}
\author{Jean-Simon Pacaud Lemay}
\begin{document}
\allowdisplaybreaks

\maketitle
%\section{}
%\subsection{}

%\tableofcontents
%\newpage

%%%%%%%%%%%%%%%%%%%%%%%%%%%%%%%%%%%%%%%%%%%%%%%%%%%%%%%%%%%%%%%%%%%%%

\begin{abstract} Differential categories were introduced by Blute, Cockett, and Seely as categorical models of differential linear logic and have since lead to abstract formulations of many notions involving differentiation such as the directional derivative, differential forms, smooth manifolds, De Rham cohomology, etc. In this paper we study the generalization of differential algebras to the context of differential categories by introducing $\mathsf{T}$-differential algebras, which can be seen as special cases of Blute, Lucyshyn-Wright, and O'Neill's notion of $\mathsf{T}$-derivations. As such, $\mathsf{T}$-differential algebras are axiomatized by the chain rule and as a consequence we obtain both the higher-order Leibniz rule and the Fa\`a di Bruno formula for the higher-order chain rule. We also construct both free and cofree $\mathsf{T}$-differential algebras for suitable codifferential categories and discuss power series of $\mathsf{T}$-algebras.
\end{abstract}

\subparagraph*{Acknowledgements.} The author would like to thank Kellogg College, the Clarendon Fund, and the Oxford-Google DeepMind Graduate Scholarship for financial support. 

\section{Introduction}

Differential linear logic \cite{ehrhard2003differential, ehrhard2006differential} is an extension of the exponential fragment of linear logic which includes a differentiation inference rule for the exponential modality. Differential categories \cite{blute2006differential} were then introduced to give categorical models of differential linear logic. In particular, a differential category has a natural transformation $\mathsf{d}$, called the {\em deriving transformations} (Definition \ref{diffcatdef}), whose axioms are based on basic properties of the derivative from differential calculus such as the Leibniz rule (also known as the product rule) and the chain rule. Differential category has now been well studied and there are many interesting examples throughout the literature \cite{blute2015cartesian, blute2010convenient, cockett2017there, cockett_lemay_2018, ehrhard2017introduction, fiore2007differential}. Differential categories have also lead to abstract formulations of several notions of differentiation such as, to list a few, the directional derivative \cite{blute2009cartesian}, K{\"a}hler differentials \cite{blute2011kahler, blute2015derivations}, differential forms \cite{cruttwell2013forms}, smooth manifolds \cite{cockett2014differential}, and De Rham cohomology \cite{Cruttwell2018, o2017smoothness}. Therefore, if the theory of differential categories wishes to champion itself as the axiomatization of the fundamentals of differentiation: differential algebras should fit naturally in this story. 

The theory of differential algebra \cite{kaplansky1957introduction, kolchin1973differential, ritt1950differential} studies algebraic objects (such as rings, algebras, field, etc.) equipped with a {\em derivation}. Briefly, for a commutative ring $R$, a differential algebra is a commutative $R$-algebra $A$ equipped with an $R$-linear endomorphism $\mathsf{D}: A \to A$ which satisfies the Leibniz rule $\mathsf{D}(ab)=a\mathsf{D}(b) + \mathsf{D}(a)b$ (see Section \ref{diffalgsec}). There are both free differential algebras \cite{guo2008differential, kolchin1973differential} and cofree differential algebras \cite{guo2008differential, keigher1975adjunctions}, the latter of which are also known as Hurwitz series rings \cite{keigher1997ring}. Due to the endomorphism nature of the derivation, one may also consider higher order derivatives to obtain a higher-order Leibniz rule and study linear differential equations \cite{keigher2000hurwitz, van2012galois}.

Originally it was hoped that the monad induced by free differential algebras would provide a differential category structure on the opposite category of $R$-modules. However, quite the opposite turned out to be true. Indeed, it was shown in \cite[Theorem 8.6]{blute2018differential} that free differential algebras do {\em NOT} provide such a differential category (unless $R$ is trivial). In order to properly introduce differential algebras to the story of differential categories, we must consider derivations evaluated in modules instead. A derivation of a commutative $R$-algebra $A$ evaluated in an $A$-module $M$ \cite{loday2013cyclic} is an $R$-linear map ${\mathsf{D}: A \to M}$ satisfying an appropriate version of the Leibniz rule. These types of derivations are particularly important in the study of differential forms and tangent spaces \cite{lee2009manifolds} and (algebraic) De Rham cohomology \cite{weibel1995introduction}. As every algebra can be seen as a module over itself, differential algebras can be equivalently defined as algebras with derivations evaluated in themselves. 

Derivations valued in modules over a commutative algebra have already been generalized to the context of a codifferential category (the dual of a differential category) by Blute, Lucyshyn-Wright, and O'Neill with the notion of $\mathsf{T}$-derivation \cite{blute2015derivations} (where $\mathsf{T}$ is the monad of a codifferential category). Briefly, every $\mathsf{T}$-algebra \cite{mac2013categories} (also known as an Eilenberg-Moore algebra) in a codifferential category comes equipped with a commutative monoid structure (see (\ref{Talgmon})) and therefore we can consider modules over them. A $\mathsf{T}$-derivation \cite[Definition 4.12]{blute2015derivations} of a $\mathsf{T}$-algebra $A$ evaluated in one its modules $M$ is a map ${\mathsf{D}: A \to M}$ which in this setting satisfies a {\em chain rule} like identity. Similarly to the classical case, a $\mathsf{T}$-differential algebra in a codifferential category will be a $\mathsf{T}$-algebra $A$ equipped with a $\mathsf{T}$-derivations of type $\mathsf{D}: A \to A$. $\mathsf{T}$-differential algebras are the appropriate generalization of differential algebras in codifferential categories.

\subsection*{Main Definitions and Results}

The main concept of study in this paper is that of $\mathsf{T}$-differential algebras for codifferential categories which are defined in Definition \ref{Tdiffalgdef} and whose basic properties are studied in Section \ref{Tdiffalgsec}. In particular, Proposition \ref{Tdalgprop} states that every $\mathsf{T}$-differential algebra is a differential algebra in the classical sense and therefore satisfies a higher-order Leibniz rule. Furthermore, $\mathsf{T}$-differential algebras also satisfy a Fa\`a di Bruno formula for a higher-order chain rule (\ref{der4}). Proposition \ref{mapTalg} and Corollary \ref{endoTalg} address $\mathsf{T}$-derivations of free $\mathsf{T}$-algebras and provide some examples of $\mathsf{T}$-differential algebras which all codifferential categories have. In Section \ref{exsec} we discuss three examples of codifferential categories and their $\mathsf{T}$-differential algebras. 

Section \ref{freesec} and Section \ref{cofreesec} are dedicated to constructing free and cofree $\mathsf{T}$-differential algebras respectively in suitable codifferential categories. Both constructions are analogous to the classical constructions of (co)free differential algebras found in \cite{guo2008differential}. Theorem \ref{freethm} states that for a codifferential category with countable coproducts, its category of $\mathsf{T}$-differential algebras is monadic over the codifferential category. While Theorem \ref{cofreethm} states that for a codifferential category with countable products, the category of $\mathsf{T}$-differential algebras is comonadic over the category of $\mathsf{T}$-algebras. The difficulty in the construction of cofree $\mathsf{T}$-differential algebras is providing the appropriate $\mathsf{T}$-algebra structure (Proposition \ref{omegatalg}) such that its induced commutative monoid structure is precisely that of the Hurwitz series ring (Lemma \ref{hurwitzmult}). Power series over $\mathsf{T}$-algebras are also discussed in Section \ref{cofreesec} and in particular for codifferential categories enriched over $\mathbb{Q}_{\geq 0}$-modules, power series over $\mathsf{T}$-algebras are isomorphic to cofree $\mathsf{T}$-differential algebras (Proposition \ref{powerH}). 

This paper concludes with a discussion on potential future work on $\mathsf{T}$-differential algebras including linear differential equations for $\mathsf{T}$-differential algebras, weighted $\mathsf{T}$-differential algebras, and generalizing Rota-Baxter algebras \cite{guo2012introduction} (the integration counterpart of differential algebras). 

\subparagraph*{Conventions:} In this paper we use diagrammatic order for composition which means that the composite map $f;g$ is the map which first does $f$ then $g$. For simplicity, we allow ourselves to work in a symmetric \emph{strict} monoidal category \cite{mac2013categories}, that is, we will suppress the unit and associativity isomorphisms. For a symmetric monoidal category we use $\otimes$ for the tensor product, $K$ for the monoidal unit, and ${\sigma: A \otimes B \to B \otimes A}$ for the symmetry isomorphism. 

\section{Differential Algebras}\label{diffalgsec}

It may be useful for the reader to review the definition of a differential algebra in the classical sense and to consider some basic examples. In Section \ref{Tdiffalgsec} we will generalize the definition and properties of differential algebras to the context of codifferential categories. For more details about the theory of differential algebras, their properties, and their applications throughout various fields, we invite the curious reader to see \cite{crespo2011algebraic, kaplansky1957introduction, kolchin1973differential, van2012galois, ritt1950differential}. 

Throughout this section, let $R$ be a commutative ring. 

\begin{definition}\label{classicaldef} An \textbf{$R$-differential algebra} (of weight $0$) is a commutative $R$-algebra $A$ equipped with a \textbf{derivation}, that is, an $R$-linear morphism $\mathsf{D}: A \to A$ which satisfies the \textbf{Leibniz rule}: 
\[\mathsf{D}(ab)=a\mathsf{D}(b) + \mathsf{D}(a)b\] \end{definition}

\begin{example} \normalfont Trivially, any commutative $R$-algebra $A$ is an $R$-differential algebra with derivation taken to be the zero map $0: A \to A$. 
\end{example}

\begin{example} \normalfont For any commutative $R$-algebra $A$, the polynomial ring in one variable $A[x]$ is an $R$-differential algebra with derivation $\mathsf{D}$ defined as the standard differentiation of polynomials:
\begin{equation}\label{classicalpoly}\begin{gathered} \mathsf{D}\left( \sum\limits^n_{k=0} a_k x^k\right) = \sum\limits^n_{k=1} a_k n x^{k-1}  \end{gathered}\end{equation}
Similarly, the power series ring $A\llbracket x \rrbracket$ is an $R$-differential algebra by extending the definition of the derivation $\mathsf{D}$ to formal power series: 
\begin{equation}\label{classicalpower}\begin{gathered} \mathsf{D}\left( \sum\limits^\infty_{k=0} a_k x^k\right) = \sum\limits^\infty_{k=1} a_k n x^{k-1} \end{gathered}\end{equation}
In Section \ref{cofreesec} we generalize the notion of power series for codifferential categories with countable products. 
\end{example}

\begin{example}\label{cinfex} \normalfont Derivations in this differential algebra sense provide an alternative algebraic description of vector fields. Let $M$ be a smooth real manifold and consider the commutative $\mathbb{R}$-algebra $\mathcal{C}^\infty(M)= \lbrace f: M \to \mathbb{R} \vert~ f \text{ smooth} \rbrace$. Then $\mathbb{R}$-derivations of $\mathcal{C}^\infty(M)$ are in bijective correspondence with smooth vector fields of $M$ \cite[Theorem 2.72]{lee2009manifolds}. Briefly, for every $\mathbb{R}$-derivation $\mathsf{D}: \mathcal{C}^\infty(M) \to \mathcal{C}^\infty(M)$ there exists a unique smooth vector field $\mathsf{v}: M \to \mathcal{T}(M)$, where $\mathcal{T}(M)$ is the tangent bundle of $M$, such that for every smooth map $f \in \mathcal{C}^\infty(M)$, $\mathsf{D}(f)$ is defined as follows: 
\begin{equation}\label{vecdiff}\begin{gathered}\mathsf{D}(f) := \xymatrixcolsep{5pc}\xymatrix{M \ar[r]^-{\mathsf{v}}& \mathcal{T}(M) \ar[r]^-{\mathcal{T}(f)} & \mathcal{T}(\mathbb{R}) \cong \mathbb{R} \times \mathbb{R} \ar[r]^-{\pi_1} & \mathbb{R}  
  } \end{gathered}\end{equation}
 where $\pi_1: \mathbb{R} \times \mathbb{R} \to \mathbb{R}$ is the \emph{second} projection and $\mathcal{T}(f)$ is the tangent map of $f$ (as defined in \cite[Definition 2.50]{lee2009manifolds}). \end{example}

\begin{example}\label{freediff} \normalfont For an $R$-module $M$ there exists an $R$-differential algebra $\mathsf{Diff}(M)$ which is the free $R$-differential algebra over $M$ \cite{guo2008differential, kolchin1973differential}. If $M$ is a free $R$-module with basis set $X= \lbrace x_1, x_2, \hdots \rbrace$, then $\mathsf{Diff}(M) \cong R[X \times \mathbb{N}]$ (the polynomial ring over the set $X \times \mathbb{N})$ with derivation $\mathsf{D}$ defined on monomials as follows:
\begin{equation}\label{}\begin{gathered} \mathsf{D}\left( (x_1, n_1) \hdots (x_m, n_m) \right) = \sum\limits^m_{k=1} (x_1,n_1) \hdots (x_k, n_k +1) \hdots (x_n, n_m) \end{gathered}\end{equation}
For arbitrary $M$, an explicit description of $\mathsf{Diff}(M)$ can be found in \cite[Example 8.3]{blute2018differential}. In Section \ref{freesec} we generalize this construction for codifferential categories with countable coproducts. 
\end{example}

\begin{example}\label{cofreediff} \normalfont For a commutative $R$-algebra $A$ there exists an $R$-differential algebra $\mathsf{H}(A)$, known as the \textbf{Hurwitz series algebra} over $A$ \cite{keigher1997ring}, which is the cofree $R$-differential algebra over $A$ \cite{guo2008differential, keigher1975adjunctions}. Explicitly, the underlying $R$-module of $\mathsf{H}(A)$ is the set of functions $f: \mathbb{N} \to A$ with point-wise addition and scalar multiplication, but whose commutative algebra structure is defined as follows: 
\begin{equation}\label{hurwtizproduct}\begin{gathered} (f \cdot g)(n) := \sum\limits^n_{k=0} \binom{n}{k} f(k)g(n-k) \quad \quad \mathsf{u}(n) := \begin{cases} 1 & \text{when $n=0$} \\
0 & o.w.\end{cases} \end{gathered}\end{equation}
The derivation $\mathsf{D}$ is simply shifting by $1$:
\begin{equation}\label{}\begin{gathered} \mathsf{D}(f)(n) := f(n+1) \end{gathered}\end{equation}
In particular, when $\mathbb{Q} \subset R$, there is an isomorphism of $R$-differential algebras $\mathsf{H}(A) \cong A\llbracket x \rrbracket$. In Section \ref{cofreesec} we generalize Hurwitz series algebras for codifferential categories with countable products. 
\end{example}

Two important generalizations of differential algebras are derivations evaluated in (bi)modules \cite{loday2013cyclic} and weighted differential algebras \cite{guo2008differential}. The former was discussed in the introduction, while the later will be discussed briefly in the conclusion. 

\section{Algebra Modalities}
In this section we review the notion of algebra modalities, a key ingredient for codifferential category structure. If only to introduce notation, we first quickly review the definitions of monads, algebras of monads, and commutative monoids. For a more thorough introductions to these concepts, we refer the reader to \cite{mac2013categories}.

\begin{definition} A \textbf{monad} on a category $\mathbb{X}$ is a triple $(\mathsf{T}, \mu, \eta)$ consisting of an functor ${\mathsf{T}: \mathbb{X} \to \mathbb{X}}$ and two natural transformations $\mu: \mathsf{T}\mathsf{T}A \to \mathsf{T}A$ and $\eta: A \to \mathsf{T}A$ such that the following diagrams commute:
      \begin{equation}\label{monadeq}\begin{gathered} \xymatrixcolsep{5pc}\xymatrix{
        \mathsf{T}  A  \ar[r]^-{\mathsf{T}(\eta)} \ar[d]_-{\eta} \ar@{=}[dr]^-{}& \mathsf{T} \mathsf{T} A \ar[d]^-{\mu}  & \mathsf{T} \mathsf{T} \mathsf{T} A  \ar[r]^-{\mu} \ar[d]_-{\mathsf{T}(\mu)} & \mathsf{T} \mathsf{T} A  \ar[d]^-{\mu}\\
        \mathsf{T} \mathsf{T} A \ar[r]_-{\mu} & \mathsf{T} A  & \mathsf{T} \mathsf{T} A \ar[r]_-{\mu} & \mathsf{T} A}\end{gathered}\end{equation} 
A \textbf{$\mathsf{T}$-algebra} is a pair $(A, \nu)$ consisting of an object $A$ and a map $\nu: \mathsf{T} A \to A$ of $\mathbb{X}$ such that the following diagrams commute:
\begin{equation}\label{Talgdef}\begin{gathered} \xymatrixcolsep{5pc}\xymatrix{
        A  \ar[r]^-{\eta} \ar@{=}[dr]^-{}& \mathsf{T} A \ar[d]^-{\nu}  &  \mathsf{T} \mathsf{T} A  \ar[r]^-{\mu} \ar[d]_-{\mathsf{T}(\nu)} & \mathsf{T} A  \ar[d]^-{\nu}\\
        &  \mathsf{T} A  & \mathsf{T} A \ar[r]_-{\nu} & A}\end{gathered}\end{equation}
A \textbf{$\mathsf{T}$-algebra morphism} $f: (A, \nu)\to (B, \nu^\prime)$ is a map $f: A\to B$ of $\mathbb{X}$ such that the following diagrams commute: 
\begin{equation}\begin{gathered} \xymatrixcolsep{5pc}\xymatrix{
       \mathsf{T} A  \ar[r]^-{\nu} \ar[d]_-{\mathsf{T}(f)} & A \ar[d]^-{f} \\
        \mathsf{T} B \ar[r]_-{\nu^\prime} & B} \end{gathered}\end{equation}     
The category of $\mathsf{T}$-algebras and $\mathsf{T}$-algebra morphisms is denoted $\mathbb{X}^\mathsf{T}$ and is known as the \textbf{Eilenberg-Moore category of algebras of the monad} $(\mathsf{T}, \mu, \eta)$. 
\end{definition}

$\mathsf{T}$-algebras of the form $(\mathsf{T}A, \mu)$ are known as \textbf{free $\mathsf{T}$-algebras}. 

\begin{definition} In a symmetric monoidal category, a \textbf{commutative monoid} is a triple $(A, \mathsf{m}, \mathsf{u})$ consisting of an object $A$, a map ${\mathsf{m}: A \otimes A\to A}$ called the multiplication, and a map $\mathsf{u}: K\to A$ called the unit, such that the following diagrams commute: 
      \begin{equation}\label{monoideq}\begin{gathered} \xymatrixcolsep{3pc}\xymatrix{
        A  \ar[r]^-{\mathsf{u} \otimes 1} \ar[d]_-{1 \otimes \mathsf{u}} \ar@{=}[dr]^-{}& A \otimes A \ar[d]^-{\mathsf{m}}  & A \otimes A \otimes A \ar[r]^-{\mathsf{m} \otimes 1} \ar[d]_-{1 \otimes \mathsf{m}} & A \otimes A  \ar[d]^-{\mathsf{m}} & A \otimes A \ar[dr]_-{\mathsf{m}}\ar[r]^-{\sigma} & A \otimes A \ar[d]^-{\mathsf{m}}\\
        A \otimes A \ar[r]_-{\mathsf{m}} & A  & A \otimes A \ar[r]_-{\mathsf{m}} & A & & A}\end{gathered}\end{equation}
A \textbf{monoid morphism}  $f: (A, \mathsf{m}, \mathsf{u}) \to (B, \mathsf{m}^\prime, \mathsf{u}^\prime)$ is a map $f: A \to B$ such that the following diagrams commute: 
\begin{equation}\label{monmorph}\begin{gathered} \xymatrixcolsep{5pc}\xymatrix{A \otimes A \ar[d]_-{\mathsf{m}} \ar[r]^-{f \otimes f} & B \otimes B \ar[d]^-{\mathsf{m}^\prime} & K \ar[dr]_-{\mathsf{u}^\prime}\ar[r]^-{\mathsf{u}} & A \ar[d]^-{f} \\
A \ar[r]_-{f} & B && B 
  } \end{gathered}\end{equation}  
\end{definition}

Algebra modalities are monads such that every free $\mathsf{T}$-algebra comes equipped with a commutative monoid structure. 

\begin{definition}\label{algmoddef} An \textbf{algebra modality} \cite{blute2006differential, blute2015derivations} on a symmetric monoidal category is a quintuple $(\mathsf{T}, \mu, \eta, \mathsf{m}, \mathsf{u})$ consisting of a monad $(\mathsf{T}, \mu, \eta)$, a natural transformation ${\mathsf{m}: \mathsf{T}(A) \otimes \mathsf{T}(A) \to \mathsf{T}(A)}$ and a natural transformation $\mathsf{u}: K \to \mathsf{T}(A)$, such that for each object $A$, the triple $(\mathsf{T}(A), \mathsf{m}, \mathsf{u})$ is a commutative monoid and $\mu: (\mathsf{T}^2(A), \mathsf{m}, \mathsf{u}) \to (\mathsf{T}(A), \mathsf{m}, \mathsf{u})$ is a monoid morphism. 
\end{definition}

Notice that for an algebra modality, requiring that $\mathsf{m}$ and $\mathsf{u}$ be natural transformations is equivalent to asking that for each map $f: A \to B$, the map $\mathsf{T}(f): \mathsf{T}(A) \to \mathsf{T}(B)$ is a monoid morphism. We discuss some examples of algebra modalities in Section \ref{exsec}, while many other examples can be found in \cite{blute2006differential, cockett_lemay_2018}. 

Every $\mathsf{T}$-algebra $(A, \nu)$ of an algebra modality comes equipped with a commutative monoid structure \cite[Theorem 2.12]{blute2015derivations} where the multiplication $\mathsf{m}^\nu: A \otimes A \to A$ and the unit $\mathsf{u}^\nu: K \to A$ are defined as follows: 
\begin{equation}\label{Talgmon}\begin{gathered} \mathsf{m}^\nu := \xymatrixcolsep{5pc}\xymatrix{A \otimes A \ar[r]^-{\eta \otimes \eta} & \mathsf{T}(A) \otimes \mathsf{T}(A) \ar[r]^-{\mathsf{m}} & \mathsf{T}(A) \ar[r]^-{\nu} & A}  \end{gathered}\end{equation} 
  \begin{equation}\label{Talgmonu}\begin{gathered} \mathsf{u}^\nu :=\xymatrixcolsep{5pc}\xymatrix{K \ar[r]^-{\mathsf{u}} & \mathsf{T}(A) \ar[r]^-{\nu} & A
  } \end{gathered}\end{equation} 
Notice that since $\mu$ is a monoid morphism, when applying this construction to a free $\mathsf{T}$-algebra $(\mathsf{T}(A), \mu)$ we recover $\mathsf{m}$ and $\mathsf{u}$, that is, $\mathsf{m}^\mu=\mathsf{m}$ and $\mathsf{u}^\nu=\mathsf{u}$. Furthermore by naturality of $\mathsf{m}$ and $\mathsf{u}$, every $\mathsf{T}$-algebra morphism becomes a monoid morphism on the induced monoid structures. In particular for every $\mathsf{T}$-algebra $(A, \nu)$, since $\nu: (A, \nu) \to (\mathsf{T}A, \mu)$ is a $\mathsf{T}$-algebra morphism, we have that $\nu: (\mathsf{T}(A), \mathsf{m}, \mathsf{u}) \to (A, \mathsf{m}^\nu, \mathsf{u}^\nu)$ is a monoid morphism. 

We finish this section with a useful natural transformation for algebra modalities. 

\begin{definition}\label{dcircdef} For an algebra modality $(\mathsf{T}, \mu, \eta, \mathsf{m}, \mathsf{u})$, its \textbf{coderiving transformaiton} \cite{cockett_lemay_2018} is the natural transformation $\mathsf{d}^\circ: \mathsf{T}(A) \otimes A \to \mathsf{T}(A)$ defined as follows: 
\begin{equation}\label{coderive}\begin{gathered} \xymatrixcolsep{5pc}\xymatrix{ \mathsf{T}(A) \otimes A \ar[r]^-{1 \otimes \eta} & \mathsf{T}(A) \otimes \mathsf{T}(A) \ar[r]^-{\mathsf{m}} & \mathsf{T}(A) 
  } \end{gathered}\end{equation}
\end{definition}

\noindent For a list of identities which the coderiving transformation satisfies see \cite[Proposition 2.1]{cockett_lemay_2018}.

\section{Codifferential Categories}

In this section we review differential categories \cite{blute2006differential}, or rather the dual notion of \textbf{codifferential categories} following \cite{blute2015derivations}. Differential categories were introduced for linear logic and therefore study coalgebra modalities (the dual notion of algebra modalities). We have chosen to work with codifferential categories as differential algebras are more natural in this setting. Of course by dualizing all the definitions and results, one could also study the dual notion of differential algebras in differential categories. 

We begin by recalling the definition of an additive symmetric monoidal category. In order to stay consistent with the terminology used in the differential category literature, here we mean ``additive'' in the Blute, Cockett, and Seely sense of the term \cite{blute2006differential}, that is, to mean enriched over commutative monoids. In particular, we do not assume negatives nor do we assume biproducts, which differs from other definitions of an additive category found in the literature \cite{mac2013categories}. 

\begin{definition} An \textbf{additive category} is a commutative monoid enriched category, that is a category in which each hom-set 
is a commutative monoid, with addition operation $+$ and zero  $0$, and in which composition preserves the additive structure, that is:
\begin{equation}\label{}\begin{gathered} k;(f\!+\!g);h\!=k;f;h\!+\!k;g; \quad \quad 0;f=0=f;0 
 \end{gathered}\end{equation}
\noindent An \textbf{additive symmetric monoidal category} is an additive category with a tensor product which is compatible with the additive structure in the sense that:
\begin{equation}\label{}\begin{gathered} k \otimes (f\!+\!g)\otimes h\!= \!k \otimes f\otimes h \!+ \!k \otimes g\otimes h \quad \quad 0\otimes h\!=\!0\end{gathered}\end{equation}
\end{definition}

In any additive category there is a notion of ``scalar multiplication'' of maps by the natural numbers $\mathbb{N}$. The scalar multiplication of a map $f: A \to B$ by $n \in \mathbb{N}$, is the map $n \cdot f: A \to B$ defined by summing $n$ copies of $f$ together, that is, $n \cdot f = f + f + \hdots + f$. If $n=0$, then $0 \cdot f =0$. Furthermore, for additive symmetric monoidal categories, one then has that $(n \cdot f ) \otimes g= n \cdot (f \otimes g)=f \otimes (n \cdot g)$. Later on we will be scalar multiplying by binomial coefficients for the higher order Leibniz rule and chain rule (Proposition \ref{Tdalgprop}) and also in constructing cofree $\mathsf{T}$-differential algebras (Section \ref{cofreesec}). 

\begin{definition}\label{diffcatdef} A \textbf{codifferential category} \cite{blute2015derivations, blute2006differential} is an additive symmetric monoidal category with an algebra modality $(\mathsf{T}, \mu, \eta, \mathsf{m}, \mathsf{u})$ which comes equipped with a \textbf{deriving transformation}, that is, a natural transformation $\mathsf{d}: \mathsf{T}(A) \to \mathsf{T}(A) \otimes A$ such that the following equalities hold: 
\begin{enumerate}[{\bf [d.1]}]
\item Constant Rule: $\mathsf{u};\mathsf{d} = 0$
\item Leibniz Rule: $\mathsf{m}; \mathsf{d} = (\mathsf{d} \otimes 1);(1 \otimes \sigma);(\mathsf{m} \otimes 1) + (1 \otimes \mathsf{d});(\mathsf{m} \otimes 1) $
\item Linear Rule: $\eta; \mathsf{d} = \mathsf{u} \otimes 1$
\item Chain Rule: $\mu; \mathsf{d}= \mathsf{d}; (\mu \otimes \mathsf{d}); (\mathsf{m} \otimes 1)$
\item Interchange Rule: $\mathsf{d}; (\mathsf{d}\otimes 1) = \mathsf{d}; (\mathsf{d}\otimes 1); (1 \otimes \sigma)$
\end{enumerate}
\end{definition}

The general intuition for codifferential categories is probably best understood by the studying Example \ref{symex} which arises from differentiating polynomials. Certain other examples of codifferential categories are also discussed in Section \ref{exsec}, while many other examples of (co)differential categories can also be found in \cite{blute2018differential, blute2006differential, cockett_lemay_2018, ehrhard2017introduction}. For an in-depth discussion on the interpretation of these axioms and (co)differential categories in general, we refer to reader to the original paper on differential categories \cite{blute2006differential}. It should be noted that the interchange rule {\bf [d.5]} was not part of the definition in \cite{blute2006differential} but was later added to ensure that the coKleisli category of a differential category was a Cartesian differential category \cite{blute2009cartesian}. 

\section{$\mathsf{T}$-Differential Algebras}\label{Tdiffalgsec}

In this section we introduce $\mathsf{T}$-differential algebras which generalizes differential algebras for codifferential categories. As discussed in the introduction, $\mathsf{T}$-differential algebras are special cases of Blute, Lucyshyn-Wright, and O'Neill's notion of $\mathsf{T}$-derivations \cite{blute2015derivations}. 

\begin{definition}\label{Tdiffalgdef} In a codifferential category with algebra modality $(\mathsf{T}, \mu, \eta, \mathsf{m}, \mathsf{u})$ and deriving transformation $\mathsf{d}$, a \textbf{$\mathsf{T}$-differential algebra} is a triple $(A, \nu, \mathsf{D})$ consisting of a $\mathsf{T}$-algebra $(A, \nu)$ and a map $\mathsf{D}: A \to A$ such that the following diagram commutes:
\begin{equation}\label{Tdiffalg}\begin{gathered} \xymatrixcolsep{5pc}\xymatrix{\mathsf{T}(A) \ar[d]_-{\mathsf{d}} \ar[rr]^-{\nu} && A \ar[d]^-{\mathsf{D}} \\
  \mathsf{T}(A)\otimes A \ar[r]_-{\nu \otimes \mathsf{D}} & A \otimes A \ar[r]_-{\mathsf{m}^\nu} & A  
  }  \end{gathered}\end{equation}
Equivalently, $\mathsf{D}: A \to A$ is a $\mathsf{T}$-derivation in the sense of \cite[Definition 4.12]{blute2015derivations}. 
\end{definition}

The intuition here is that $\mathsf{T}$-differential algebras are axiomatized by a chain rule, involving the $\mathsf{T}$-algebra structure $\nu$ rather than simply being axiomatized by the Leibniz rule. However, every $\mathsf{T}$-differential algebra is a differential algebra in the classical sense, that is, $\mathsf{D}$ satisfies the Leibniz rule (\ref{der2}) and the constant rule (\ref{der1}). $\mathsf{T}$-differential algebras also have a notion of higher order derivatives, which is not necessarily the case for arbitrary $\mathsf{T}$-derivations $\mathsf{D}: A \to M$ in the sense of \cite{blute2015derivations}. Explicitly, we have for each $n \in \mathbb{N}$ the map $\mathsf{D}^n : A \to A$ which applies the derivative $n$ times, $\mathsf{D}^n := \mathsf{D}; \hdots; \mathsf{D}$, and where by convention $\mathsf{D}^0=1_A$. We obtain both a higher-order Leibniz rule (\ref{der3}) and a Fa\`a di Bruno formula (\ref{der4}) for the higher-order chain rule for $\mathsf{T}$-differential algebras. 

\begin{proposition}\label{Tdalgprop} Let $(A, \nu, \mathsf{D})$ be a $\mathsf{T}$-differential algebra. Then $\mathsf{D}$ satisfies the following identities: 
\begin{enumerate}[{\em (i)}]
\item Constant Rule:
\begin{equation}\label{der1}\begin{gathered} \mathsf{u}^\nu; \mathsf{D}= 0
 \end{gathered}\end{equation}
\item Leibniz Rule: 
\begin{equation}\label{der2}\begin{gathered} \mathsf{m}^\nu; \mathsf{D}= (1 \otimes \mathsf{D}); \mathsf{m}^\nu +  (\mathsf{D} \otimes 1); \mathsf{m}^\nu
 \end{gathered}\end{equation}
 \item Higher-Order Leibniz Rule (for all $n \in \mathbb{N}$): 
\begin{equation}\label{der3}\begin{gathered} \mathsf{m}^\nu; \mathsf{D}^{n} =  \sum\limits^{n}_{k=0} \binom{n}{k} \cdot (\mathsf{D}^k \otimes \mathsf{D}^{n-k}); \mathsf{m}^\nu
 \end{gathered}\end{equation}
\item Fa\`a di Bruno Rule  (for all $n \in \mathbb{N}$): 
\begin{equation}\label{der4}\begin{gathered} \nu; \mathsf{D}^{n+1} =  \sum\limits^{n}_{k=0} \binom{n}{k} \cdot   \mathsf{d}; (\nu \otimes 1); (\mathsf{D}^k \otimes \mathsf{D}^{n-k+1}); \mathsf{m}^\nu \end{gathered}\end{equation}
\end{enumerate}
\end{proposition} 
\begin{proof} The constant rule and the Leibniz rule follow mostly from the axioms of the same name for the deriving transformation.  \\Ê\\
{\em (i) Constant Rule}: 
\begin{align*}
\mathsf{u}^\nu; \mathsf{D} = \mathsf{u}; \nu; \mathsf{D} = \mathsf{u};  \mathsf{d}; (\nu \otimes \mathsf{D}); \mathsf{m}^\nu = 0; (\nu \otimes \mathsf{D}); \mathsf{m}^\nu= 0 
\end{align*}
{\em (ii) Leibniz Rule}: We first show that the deriving transformation satisfies the following identity
\begin{equation}\label{dpoly}\begin{gathered} 
(\eta \otimes \eta);\mathsf{m}; \mathsf{d} = (\eta \otimes 1) + (1 \otimes \eta);\sigma
\end{gathered}\end{equation}
by using both the Leibniz rule \textbf{[d.2]} and the Linear rule \textbf{[d.3]}. 
\begin{align*}
(\eta \otimes \eta);\mathsf{m}; \mathsf{d} &=~ (\eta \otimes \eta);\left( (\mathsf{d} \otimes 1);(1 \otimes \sigma);(\mathsf{m} \otimes 1) + (1 \otimes \mathsf{d});(\mathsf{m} \otimes 1) \right) \tag{Leibniz Rule \textbf{[d.2]}} \\
&=~ (\eta \otimes \eta);(\mathsf{d} \otimes 1);(1 \otimes \sigma);(\mathsf{m} \otimes 1) + (\eta \otimes \eta);(1 \otimes \mathsf{d});(\mathsf{m} \otimes 1) \\
&=~ (1 \otimes \eta);(\mathsf{u} \otimes 1 \otimes 1);(1 \otimes \sigma);(\mathsf{m} \otimes 1) + (1 \otimes \eta);(1 \otimes \mathsf{u} \otimes 1);(\mathsf{m} \otimes 1) \tag{Linear Rule \textbf{[d.3]}}  \\
&=~ (\eta \otimes 1) + (1 \otimes \eta);\sigma \tag{\ref{monadeq}} \\
\end{align*}
Now we can show that $\mathsf{D}$ satisfies the Leibniz rule: 
\begin{align*}
\mathsf{m}^\nu; \mathsf{D} &=~ (\eta \otimes \eta);\mathsf{m}; \nu; \mathsf{D} \tag{\ref{Talgmon}} \\
&=~ (\eta \otimes \eta);\mathsf{m}; \mathsf{d}; (\nu \otimes \mathsf{D}); \mathsf{m}^\nu \tag{\ref{Tdiffalg}} \\
&=~ \left((\eta \otimes 1) + (1 \otimes \eta);\sigma \right)(\nu \otimes \mathsf{D}); \mathsf{m}^\nu \tag{\ref{dpoly}} \\
&=~Ê(\eta \otimes 1); (\nu \otimes \mathsf{D}); \mathsf{m}^\nu + (1 \otimes \eta);\sigma;(\nu \otimes \mathsf{D}); \mathsf{m}^\nu \\
&=~Ê(\eta \otimes 1); (\nu \otimes \mathsf{D}); \mathsf{m}^\nu + (1 \otimes \eta);(\mathsf{D} \otimes \nu);\sigma; \mathsf{m}^\nu \\
&=~Ê(1 \otimes \mathsf{D}); \mathsf{m}^\nu + (\mathsf{D} \otimes 1);\sigma; \mathsf{m}^\nu \tag{\ref{Talgdef}}\\
&=~Ê(1 \otimes \mathsf{D}); \mathsf{m}^\nu + (\mathsf{D} \otimes 1); \mathsf{m}^\nu \tag{\ref{monoideq}}
\end{align*}
For the higher-order Leibniz rule and the Fa\`a di Bruno rule, we will need the following well-known identity for binomial coefficients (where $1 \leq k \leq n$):
\begin{equation}\label{binom}\begin{gathered} \binom{n}{k} + \binom{n}{k-1} = \binom{n+1}{k} \end{gathered}\end{equation} 
{\em (iii) Higher-Order Leibniz Rule}: We prove this by induction on $n$. For the base case $n=0$, the identity holds trivally. Suppose that the desired equality holds for $k \leq n$, we now show it for $n+1$: 
\begin{align*}
&\mathsf{m}^\nu; \mathsf{D}^{n+1} =~ \nu; \mathsf{D}^{n}; \mathsf{D} \\
&=~ \left(\sum\limits^{n}_{k=0} \binom{n}{k} \cdot (\mathsf{D}^k \otimes \mathsf{D}^{n-k}); \mathsf{m}^\nu \right); \mathsf{D} \tag{Induction Hypothesis}\\
&=~ \sum\limits^{n}_{k=0} \binom{n}{k} \cdot (\mathsf{D}^k \otimes \mathsf{D}^{n-k}); \mathsf{m}^\nu; \mathsf{D} \\
&=~ \sum\limits^{n}_{k=0} \binom{n}{k} \cdot (\mathsf{D}^k \otimes \mathsf{D}^{n-k}); \left((\mathsf{D} \otimes 1);\mathsf{m}^\nu + (1 \otimes \mathsf{D});\mathsf{m}^\nu \right) \tag{\ref{der2}} \\
&=~ \sum\limits^{n}_{k=0} \binom{n}{k} \cdot (\mathsf{D}^k \otimes \mathsf{D}^{n-k});(\mathsf{D} \otimes 1);\mathsf{m}^\nu \\
&~~~+ \sum\limits^{n}_{k=0} \binom{n}{k} \cdot (\mathsf{D}^k \otimes \mathsf{D}^{n-k});(1 \otimes \mathsf{D});\mathsf{m}^\nu \\
&=~ \sum\limits^{n}_{k=0} \binom{n}{k} \cdot (\mathsf{D}^{k+1} \otimes \mathsf{D}^{n-k});\mathsf{m}^\nu \\
&~~~+ \sum\limits^{n}_{k=0} \binom{n}{k} \cdot (\mathsf{D}^k \otimes \mathsf{D}^{n+1-k});\mathsf{m}^\nu \\ 
&=~ \sum\limits^{n+1}_{j=1} \binom{n}{j-1} \cdot (\mathsf{D}^{j} \otimes \mathsf{D}^{n+1-j});\mathsf{m}^\nu + \sum\limits^{n}_{k=0} \binom{n}{k} \cdot (\mathsf{D}^k \otimes \mathsf{D}^{n+1-k});\mathsf{m}^\nu \tag{Reindexing}\\ 
&=~ (1 \otimes \mathsf{D}^{n+1}); \mathsf{m}^\nu + (\mathsf{D}^{n+1} \otimes 1); \mathsf{m}^\nu  \\
&~~~+ \sum\limits^{n}_{j=1} \binom{n}{j-1} \cdot (\mathsf{D}^{j} \otimes \mathsf{D}^{n+1-j});\mathsf{m}^\nu + \sum\limits^{n}_{k=1} \binom{n}{k} \cdot (\mathsf{D}^k \otimes \mathsf{D}^{n+1-k});\mathsf{m}^\nu \\ 
&=~Ê(1 \otimes \mathsf{D}^{n+1}); \mathsf{m}^\nu + (\mathsf{D}^{n+1} \otimes 1); \mathsf{m}^\nu +  \sum\limits^{n}_{k=1} \left(\binom{n}{k} + \binom{n}{k-1}\right) \cdot (\mathsf{D}^k \otimes \mathsf{D}^{n+1-k}); \mathsf{m}^\nu \\
&=~ (1 \otimes \mathsf{D}^{n+1}); \mathsf{m}^\nu +  (\mathsf{D}^{n+1} \otimes 1); \mathsf{m}^\nu + \sum\limits^{n}_{k=1} \binom{n+1}{k} \cdot (\mathsf{D}^k \otimes \mathsf{D}^{n+1-k}); \mathsf{m}^\nu \\
&=~ \sum\limits^{n+1}_{k=0} \binom{n+1}{k} \cdot (\mathsf{D}^k \otimes \mathsf{D}^{n+1-k}); \mathsf{m}^\nu \tag{\ref{binom}}
\end{align*}
{\em (iv) Fa\`a di Bruno Rule}: We prove this by induction on $n$. The base case $n=0$ is precisely (\ref{Tdiffalg}). Suppose that the desired equality holds for $k \leq n$, we now show it for $n+1$: 
\begin{align*}
&\nu; \mathsf{D}^{n+2} = ~ \nu; \mathsf{D}^{n+1}; \mathsf{D} \\
&=~ \left( \sum\limits^{n}_{k=0} \binom{n}{k} \cdot   \mathsf{d}; (\nu \otimes 1); (\mathsf{D}^k \otimes \mathsf{D}^{n-k+1}); \mathsf{m}^\nu  \right); \mathsf{D} \tag{Induction Hypothesis}\\
&=~  \sum\limits^{n}_{k=0} \binom{n}{k} \cdot   \mathsf{d}; (\nu \otimes 1); (\mathsf{D}^k \otimes \mathsf{D}^{n-k+1}); \mathsf{m}^\nu ; \mathsf{D} \\
&=~  \sum\limits^{n}_{k=0} \binom{n}{k} \cdot   \mathsf{d}; (\nu \otimes 1); (\mathsf{D}^k \otimes \mathsf{D}^{n-k+1}); (\mathsf{D} \otimes 1);\mathsf{m}^\nu\\
&~~~ + \sum\limits^{n}_{k=0} \binom{n}{k} \cdot   \mathsf{d}; (\nu \otimes 1); (\mathsf{D}^k \otimes \mathsf{D}^{n-k+1});(1 \otimes \mathsf{D});\mathsf{m}^\nu \\
&=~  \sum\limits^{n}_{k=0} \binom{n}{k} \cdot   \mathsf{d}; (\nu \otimes 1); (\mathsf{D}^{k+1} \otimes \mathsf{D}^{n-k+1}); \mathsf{m}^\nu\\
&~~~ +\sum\limits^{n}_{k=0} \binom{n}{k} \cdot   \mathsf{d}; (\nu \otimes 1); (\mathsf{D}^k \otimes \mathsf{D}^{n+1-k+1});\mathsf{m}^\nu \\
&=~  \sum\limits^{n+1}_{j=1} \binom{n}{j-1} \cdot   \mathsf{d}; (\nu \otimes 1); (\mathsf{D}^{j} \otimes \mathsf{D}^{n+1-j+1}); \mathsf{m}^\nu \\
&~~~+\sum\limits^{n}_{k=0} \binom{n}{k} \cdot   \mathsf{d}; (\nu \otimes 1); (\mathsf{D}^k \otimes \mathsf{D}^{n+1-k+1});\mathsf{m}^\nu \tag{Reindexing} \\
&=~ \mathsf{d}; (\nu \otimes 1); (\mathsf{D}^{n+1} \otimes \mathsf{D}); \mathsf{m}^\nu +  \sum\limits^{n}_{j=1} \binom{n}{j-1} \cdot   \mathsf{d}; (\nu \otimes 1); (\mathsf{D}^{j} \otimes \mathsf{D}^{n+1-j+1}); \mathsf{m}^\nu \\
&~~~+ \mathsf{d}; (\nu \otimes 1); (1 \otimes \mathsf{D}^{n+2});\mathsf{m}^\nu +\sum\limits^{n}_{k=1} \binom{n}{k} \cdot   \mathsf{d}; (\nu \otimes 1); (\mathsf{D}^k \otimes \mathsf{D}^{n+1-k+1});\mathsf{m}^\nu \\
&=~ \mathsf{d}; (\nu \otimes 1); (\mathsf{D}^{n+1} \otimes \mathsf{D}); \mathsf{m}^\nu +  \mathsf{d}; (\nu \otimes 1); (1 \otimes \mathsf{D}^{n+2});\mathsf{m}^\nu \\
&~~~+\sum\limits^{n}_{k=1} \left( \binom{n}{k} + \binom{n}{k-1} \right) \cdot   \mathsf{d}; (\nu \otimes 1); (\mathsf{D}^k \otimes \mathsf{D}^{n+1-k+1});\mathsf{m}^\nu \\
&=~ \mathsf{d}; (\nu \otimes 1); (\mathsf{D}^{n+1} \otimes \mathsf{D}); \mathsf{m}^\nu +  \mathsf{d}; (\nu \otimes 1); (1 \otimes \mathsf{D}^{n+2});\mathsf{m}^\nu \\
&~~~+\sum\limits^{n}_{k=1} \binom{n+1}{k} \cdot   \mathsf{d}; (\nu \otimes 1); (\mathsf{D}^k \otimes \mathsf{D}^{n+1-k+1});\mathsf{m}^\nu \tag{\ref{binom}} \\
&=~Ê \sum\limits^{n+1}_{k=0} \binom{n+1}{k} \cdot   \mathsf{d}; (\nu \otimes 1); (\mathsf{D}^k \otimes \mathsf{D}^{n+1-k+1}); \mathsf{m}^\nu 
\end{align*}
\end{proof} 

A classical result for differential algebras is that the set of elements whose derivative is zero is a sub-algebra \cite{van2012galois}. A similar property is true for $\mathsf{T}$-differential algebras.  

\begin{proposition} Let $(A, \nu, \mathsf{D})$ be a $\mathsf{T}$-differential algebra and suppose that the kernel of $\mathsf{D}$ exists, that is, there exists an object $\mathsf{ker}(\mathsf{D})$ and a map $\iota: \mathsf{ker}(\mathsf{D}) \to A$ such that the following is an equalizer diagram: 
  \[  \xymatrixcolsep{5pc}\xymatrix{\mathsf{ker}(\mathsf{D}) \ar[r]^-{\iota} & A \ar@<1ex>[r]^{\mathsf{D}} \ar@<-1ex>[r]_{0} & A  
  } \]
Then kernel of $\mathsf{D}$ admits a $\mathsf{T}$-algebra structure $(\mathsf{ker}(\mathsf{D}), \tilde{\iota})$ such that $\iota: (\mathsf{ker}(\mathsf{D}), \tilde{\iota}) \to (A, \nu)$ is a monic $\mathsf{T}$-algebra morphism. 
\end{proposition} 
\begin{proof} Notice that we have the following equality: 
\begin{align*}
\mathsf{T}(\iota);\nu; \mathsf{D} &=~Ê\mathsf{T}(\iota); \mathsf{d}; (\nu \otimes  \mathsf{D}); \mathsf{m}^\nu =Ê\mathsf{d}; (\mathsf{T}(\iota) \otimes \iota); (\nu \otimes \mathsf{D}); \mathsf{m}^\nu =Ê\mathsf{d}; (\mathsf{T}(\iota) \otimes \iota); (\nu \otimes 0); \mathsf{m}^\nu = 0
\end{align*}
Therefore, $\mathsf{T}(\iota);\nu$ equalizes $0$ and $\mathsf{D}$. So by universality of the kernel, there exists a unique map $\tilde{\iota}: \mathsf{T}(\mathsf{ker}(\mathsf{D})) \to \mathsf{ker}(\mathsf{D})$ such that the following diagram commutes: 
  \[  \xymatrixcolsep{5pc}\xymatrix{\mathsf{ker}(\mathsf{D}) \ar[r]^-{\iota} & A \ar@<1ex>[r]^{\mathsf{D}} \ar@<-1ex>[r]_{0} & A  \\
  \mathsf{T}(\mathsf{ker}(\mathsf{D}) \ar[r]_-{\mathsf{T}(\iota)} \ar@{-->}[u]^-{\exists \oc ~\tilde{\iota}} & \mathsf{T}(A) \ar[u]_-{\nu}
  } \]
  To show that $(\mathsf{ker}(\mathsf{D}), \tilde{\iota})$  is a $\mathsf{T}$-algebra, notice the following equalities: 
\begin{align*}
\eta; \tilde{\iota}; \iota &=~ \eta;\mathsf{T}(\iota);\nu =Ê\iota; \eta; \nu = \iota
\end{align*}  
\begin{align*}
\mathsf{T}(\tilde{\iota}); \tilde{\iota}; \iota &=~ \mathsf{T}(\tilde{\iota});\mathsf{T}(\iota);\nu = \mathsf{T}^2(\iota);\mathsf{T}(\nu);\nu = \mathsf{T}^2(\iota);\mu;\nu =Ê\mu;\mathsf{T}(\iota);\nu = \mu;\tilde{\iota}; \iota 
\end{align*}
Since $\iota$ is monic, we have $\eta; \tilde{\iota}=1$ and $\mathsf{T}(\tilde{\iota}); \tilde{\iota}=Ê\mu;\tilde{\iota}$. Therefore, $(\mathsf{ker}(\mathsf{D}), \tilde{\iota})$ is a $\mathsf{T}$-algebra, and by definition of $\tilde{\iota}$, $\iota$ is a $\mathsf{T}$-algebra morphism. 
\end{proof} 

Next we discuss $\mathsf{T}$-differential algebra morphisms. 

\begin{definition} A \textbf{$\mathsf{T}$-differential algebra morphism} $f: (A, \nu, \mathsf{D}) \to (B, \nu^\prime, \mathsf{D}^\prime)$ is a $\mathsf{T}$-algebra morphism $f: (A, \nu) \to (B, \nu^\prime)$ such that the following diagram commutes: 
\begin{equation}\label{Tdiffmap}\begin{gathered} \xymatrixcolsep{5pc}\xymatrix{A \ar[d]_-{\mathsf{D}} \ar[r]^-{f} & B \ar[d]^-{\mathsf{D}^\prime} \\
A \ar[r]_-{f} & B
  } \end{gathered}\end{equation}
\end{definition}

For a codifferential category $\mathbb{X}$ with algebra modality $(\mathsf{T}, \mu, \eta, \mathsf{m}, \mathsf{u})$, we denote the category of $\mathsf{T}$-differential algebras and $\mathsf{T}$-differential algebra morphisms by $\mathsf{D}[\mathbb{X}^\mathsf{T}]$ (where recall that $\mathbb{X}^\mathsf{T}$ is the Eilenberg-Moore category of $\mathsf{T}$). In Section \ref{freesec} and Section \ref{cofreesec}, we will construct both free and cofree $\mathsf{T}$-differential algebras for suitable codifferential categories and in particular show that the category of $\mathsf{T}$-differential algebras is monadic and comonadic. 

We now turn our attention to $\mathsf{T}$-differential algebra structure for free $\mathsf{T}$-algebras. In this case, $\mathsf{T}$-derivations are completely determined by a certain set of maps. 

\begin{lemma} For a $\mathsf{T}$-algebra $(A, \nu)$, let $\mathcal{D}(A, \nu)$ be the set of maps $\mathsf{D}: A \to A$ such that $(A, \nu, \mathsf{D})$ is a $\mathsf{T}$-differential algebra. Then $\mathcal{D}(A, \nu)$ is a commutative monoid, that is, $0 \in \mathcal{D}(A, \nu)$ and if $\mathsf{D}_1 \in \mathcal{D}(A, \nu)$ and $\mathsf{D}_2 \in \mathcal{D}(A, \nu)$, then so is $\mathsf{D}_1+ \mathsf{D}_2 \in \mathcal{D}(A, \nu)$. 
\end{lemma} 
\begin{proof} Follows from the additive symmetric monoidal structure of a codifferential category. 
\end{proof} 

\begin{proposition}\label{mapTalg} For each object $A$ in a codifferential category $\mathbb{X}$ with algebra modality $(\mathsf{T}, \mu, \eta, \mathsf{m}, \mathsf{u})$, we have the following commutative monoid isomorphism: 
\begin{equation}\label{}\begin{gathered} \mathcal{D}(\mathsf{T}(A), \mu) \cong \mathbb{X}(A, \mathsf{T}(A))
 \end{gathered}\end{equation}
\end{proposition} 
\begin{proof} This a direct consequence of \cite[Theorem 4.15]{blute2015derivations} and therefore we will only give how to construct one from the other. If $(\mathsf{T}(A), \mu, \mathsf{D})$ is a $\mathsf{T}$-differential algebra, then precomposing by $\eta$ we obtain:
  \[  \xymatrixcolsep{5pc}\xymatrix{A \ar[r]^-{\eta} & \mathsf{T}(A) \ar[r]^-{\mathsf{D}} & \mathsf{T}(A)  
  } \]
  Conversly, given a map $f: A \to \mathsf{T}(A)$, define the map $f^\flat: \mathsf{T}(A) \to \mathsf{T}(A)$ as follows: 
    \[ f^\flat := \xymatrixcolsep{5pc}\xymatrix{\mathsf{T}(A) \ar[r]^-{\mathsf{d}} & \mathsf{T}(A) \otimes A \ar[r]^-{1 \otimes f} & \mathsf{T}(A) \otimes \mathsf{T}(A) \ar[r]^-{\mathsf{m}} &\mathsf{T}(A)  
  } \]
  Then $(\mathsf{T}(A), \mu, f^\flat)$ is a $\mathsf{T}$-differential algebra.
\end{proof} 

In particular, every endomorphism induces a $\mathsf{T}$-differential algebra. This will be important when we construct free $\mathsf{T}$-differential algebras in Section \ref{freesec}. 

\begin{corollary}\label{endoTalg} Every endomorphism $f: A \to A$ induces a $\mathsf{T}$-differential algebra $(\mathsf{T}(A), \mu, f^\sharp)$ where $f^\sharp := (f; \eta)^\flat$, or more explicitly: 
\begin{equation}\label{fsharp}\begin{gathered} f^\sharp := \xymatrixcolsep{5pc}\xymatrix{\mathsf{T}(A) \ar[r]^-{\mathsf{d}} & \mathsf{T}(A) \otimes A \ar[r]^-{1 \otimes f} & \mathsf{T}(A) \otimes A \ar[r]^-{\mathsf{d}^\circ} &\mathsf{T}(A)  } \end{gathered}\end{equation}
\end{corollary}

\begin{corollary}\label{dcircv} If $(A, \nu, \mathsf{D})$ is a $\mathsf{T}$-differential algebra, then $\nu: (\mathsf{T}A, \mu, \mathsf{D}^\sharp) \to (A, \nu, \mathsf{D})$ is a $\mathsf{T}$-differential algebra morphism. 
\end{corollary}
\begin{proof} This follows from the fact that $\nu$ is a monoid morphism and a $\mathsf{T}$-algebra structure: 
\begin{align*}
\mathsf{D}^\sharp; \nu &=~  \mathsf{d}; (1 \otimes \mathsf{D}); \mathsf{d}^\circ; \nu \\
&=~ \mathsf{d}; (1 \otimes \mathsf{D}); (1 \otimes \eta); \mathsf{m}; \nu \\
&=~ \mathsf{d}; (1 \otimes \mathsf{D}); (1 \otimes \eta); (\nu \otimes \nu); \mathsf{m}^\nu \tag{\ref{monmorph}} \\
&=~ \mathsf{d}; (\nu \otimes \mathsf{D}); \mathsf{m}^\nu \tag{\ref{Talgdef}} \\
&=~ \nu; \mathsf{D} \tag{\ref{Tdiffalg}}
\end{align*}
\end{proof} 

We finish this section with some simple examples of $\mathsf{T}$-differential algebras which all codifferential categories have.

\begin{example} \normalfont For every map $a: A \to K$, define the map $\mathsf{d}_{a}: \mathsf{T}(A) \to \mathsf{T}(A)$ as follows: 
\[\mathsf{d}_a := (a; \mathsf{u})^\flat = \mathsf{d}; (1 \otimes a)\]
By Proposition \ref{mapTalg}, $(\mathsf{T}(A), \mu, \mathsf{d}_a)$ is a $\mathsf{T}$-differential algebra. In particular for the monoidal unit $K$, $\mathsf{d}_{1_K} = \mathsf{d}$ and therefore $(\mathsf{T}(K), \mu, \mathsf{d})$ is a $\mathsf{T}$-differential algebra. \end{example}

\begin{example} \normalfont In a codifferential category, define the natural transformation $\mathsf{L}: \mathsf{T}(A) \to \mathsf{T}(A)$ as follows: 
  \[ \mathsf{L} := \xymatrixcolsep{5pc}\xymatrix{\mathsf{T}(A) \ar[r]^-{\mathsf{d}} & \mathsf{T}(A) \otimes A \ar[r]^-{\mathsf{d}^\circ} & \mathsf{T}(A)  
  } \]
This natural transformation is important for the notion of {\em antiderivatives} for codifferential categories \cite{cockett_lemay_2018}. Notice that $\mathsf{L} = 1^\sharp$ and so by Corollary \ref{endoTalg}, $(\mathsf{T}(A), \mu, \mathsf{L})$ is a $\mathsf{T}$-differential algebra. In fact, that this is a $\mathsf{T}$-differential algebra was crucial for many of the proofs in \cite[Section 6]{cockett_lemay_2018} (though the notion of $\mathsf{T}$-differential algebra was not developed at the time). 
\end{example}

\section{Examples}\label{exsec}

In this section we discuss three examples of codifferential categories and their $\mathsf{T}$-differential algebras.  

\begin{example} \normalfont \label{symex} A well-known example of a codifferential category comes from the symmetric algebra construction \cite{blute2018differential, blute2006differential, blute2015derivations} where the differential structure corresponds to polynomial differentiation. Let $R$ be a commutative ring and let $\mathsf{MOD}_\mathbb{R}$ be the category of $R$-modules and $R$-linear maps between them. The additive symmetric monoidal structure of $\mathsf{MOD}_\mathbb{R}$ is given by the standard tensor product and additive enrichment of $R$-modules. For an $R$-module $M$, define the symmetric algebra over $M$, $\mathsf{Sym}(M)$, as follows (see \cite[Section 8, Chapter XVI ]{lang2002algebra} for more details): 
$$\mathsf{Sym}(M)= \bigoplus^{\infty}_{n=0} \mathsf{Sym}^n(M)= R \oplus M \oplus \mathsf{Sym}^2(M) \oplus ... $$ 
where $\mathsf{Sym}^n(M)$ is simply the quotient of $M^{\otimes^n}$ by the tensor symmetry equalities. $\mathsf{Sym}(M)$ is the free commutative $R$-algebra over $M$. This induces a monad $\mathsf{Sym}$ on $\mathsf{MOD}_R$ which is also an algebra modality (see \cite[Proposition 2.9]{blute2006differential} for more details). The algebra modality comes equipped with a deriving transformation $\mathsf{d}: \mathsf{Sym}(M) \to ~\mathsf{Sym}(M) \otimes M$ defined on pure tensors as follows:
$$\mathsf{d}(a_1 \otimes ... \otimes a_n)= \sum_{i=1}^{n} (a_1 \otimes ... \otimes a_{i-1} \otimes a_{i+1} \otimes ... \otimes a_n) \otimes a_i$$ 
which we extend by $R$-linearity. This makes $\mathsf{MOD}_R$ a codifferential category. To see that this is in fact polynomial differentiation, one should look at free $R$-modules. For example, if $M$ is a free module with basis $X=\lbrace x_1, x_2, \hdots \rbrace$, then $\mathsf{Sym}(M) \cong R[X]$ as $R$-algebras (where $R[X]$ is the polynomial ring over the set $X$) and the deriving transformation can be expressed as taking the sum of partial derivatives:
\[\mathsf{d}(\mathsf{p}(x_1, \hdots, x_n))=  \sum_{i=1}^{n} \frac{\partial \mathsf{p}}{\partial x_i}(x_1, \hdots, x_n) \otimes x_i \]
The $\mathsf{Sym}$-algebras are precisely the commutative $R$-algebras, that is, $\mathsf{MOD}^\mathsf{Sym}_R$ is isomorphic to the category of commutative $R$-algebras. While the $\mathsf{Sym}$-differential algebras are precisely the $R$-differential algebras in the classical sense (Definition \ref{classicaldef}).  The required $\mathsf{Sym}$-derivation coherence (\ref{Tdiffalg}) expresses that classical derivations satisfy a chain rule like identity with polynomials. Indeed, for an $R$-differential algebra $A$ equipped with derivation $\mathsf{D}$ and for any polynomial $\mathsf{p}(x_1, \hdots, x_n) \in R[x_1, \hdots, x_n]$, the following equality holds: 
\begin{equation}\label{}\begin{gathered} \mathsf{D}(\mathsf{p}(a_1, \hdots, a_n)) = \sum_{i=1}^{n} \frac{\partial \mathsf{p}}{\partial x_i}(a_1, \hdots, a_n) \mathsf{D}(a_i) \quad \quad \forall a_i \in A\end{gathered}\end{equation}
which is quite similar to the classical chain rule from differential calculus: $(f \circ g)^\prime (x) = f^\prime (g(x)) g^\prime(x)$. 

Looking forward, since $\mathsf{MOD}_R$ admits both countable products and countable coproducts, this codifferential category admits both free and cofree $\mathsf{Sym}$-differential algebras as constructed in Section \ref{freesec} and Section \ref{cofreesec}. In fact, the free and cofree $\mathsf{Sym}$-differential algebras are precisely those described in Example \ref{freediff} and Example \ref{cofreediff}. 
\end{example} 

\begin{example} \normalfont \label{smoothex} A $\mathcal{C}^\infty$-ring \cite{joyce2011introduction, moerdijk2013models} is a set $A$ equipped a with a family of functions $\Phi_f: A^n \to A$ for each smooth function $f: \mathbb{R}^n \to \mathbb{R}$ satisfying certain compatibility coherences. For example, if $M$ is a smooth manifold, then $\mathcal{C}^\infty(M)$ (as defined in Example \ref{cinfex}) is a $\mathcal{C}^\infty$-ring where for a smooth map $f: \mathbb{R}^n \to \mathbb{R}$, $\Phi_f: \mathcal{C}^\infty(M)^n  \to \mathcal{C}^\infty(M)$ is defined by post-composition by $f$:  
  \[  \xymatrixcolsep{5pc}\xymatrix{\mathcal{C}^\infty(M)^n  \ar[rr]^-{\Phi_f}  & & \mathcal{C}^\infty(M) \\
  } \]
    \[ \left(\xymatrix{ M \ar[r]^-{g_1} & \mathbb{R}
  }, \hdots, \xymatrix{ M \ar[r]^-{g_n} & \mathbb{R}
  } \right) \longmapsto \left(\xymatrixcolsep{3pc}\xymatrix{ M \ar[r]^-{\langle g_1, \hdots, g_n \rangle} & \mathbb{R}^n \ar[r]^-{f} & \mathbb{R}
  } \right) \]

Every $\mathcal{C}^\infty$-ring is a commutative $\mathbb{R}$-algebra, and furthermore, for every $\mathbb{R}$-vector space $V$ there exists a free $\mathcal{C}^\infty$-ring over $V$ \cite[Theorem 3.3]{kainz1987c}, which we denote as $\mathsf{S}^\infty(V)$. This induces a monad $\mathsf{S}^\infty$ on the category of $\mathbb{R}$-vector spaces, $\mathsf{VEC}_\mathbb{R}$, which is an algebra modality and comes equipped with a deriving transformation as defined in \cite[Section 3]{blute2006differential} \footnote{A paper by Cruttwell, Lucyshyn-Wright, and the author which studies this example in more detail is in the works.}, making $\mathsf{VEC}_\mathbb{R}$ a codifferential category. This differential structure is defined by differentiating smooth functions. In particular, for $\mathbb{R}^n$, $\mathsf{S}^\infty(\mathbb{R}^n)= \mathsf{C}^\infty(\mathbb{R}^n)$ and the deriving transformation $\mathsf{d}: \mathsf{C}^\infty(\mathbb{R}^n) \to \mathsf{C}^\infty(\mathbb{R}^n) \otimes \mathbb{R}^n$ is defined as the sum of partial derivatives (similar to the previous example with polynomials): 
\begin{align*}
\mathsf{d}(f)=  \sum_{i=1}^{n} \frac{\partial f}{\partial x_i} \otimes x_i 
\end{align*}
The $\mathsf{S}^\infty$-algebras are precisely $\mathcal{C}^\infty$-rings, while $\mathsf{S}^\infty$-differential algebras are $\mathcal{C}^\infty$-rings $A$ equipped with an endomorphism $\mathsf{D}: A \to A$ which is compatible with the $\mathcal{C}^\infty$-ring structure in the following sense: 
\begin{equation}\label{}\begin{gathered}\mathsf{D}\left(\Phi_f(a_1, \hdots, a_n)\right) = \sum\limits^n_{i=1} \Phi_{\frac{\partial f}{\partial x_i}}(a_1, \hdots, a_n)\mathsf{D}(a_i) \quad \quad \forall a_i \in A \end{gathered}\end{equation}
By Proposition \ref{Tdalgprop}, every $\mathsf{S}^\infty$-differential algebra is an $\mathbb{R}$-differential algebra and these types of derivations from the point of view of K{\"a}hler differentials are discussed in \cite{dubuc19841,joyce2011introduction}. 

As discussed in Example \ref{cinfex}, for a smooth manifold $M$ there is a bijective correspondence between $\mathbb{R}$-derivations of $\mathcal{C}^\infty(M)$ (which are $\mathsf{Sym}$-derivations) and smooth vector fields of $M$ \cite[Theorem 2.72]{lee2009manifolds}. Therefore since every $\mathbb{R}$-derivation of $\mathcal{C}^\infty(M)$ is of the form (\ref{vecdiff}), it follows that every $\mathbb{R}$-derivation of $\mathcal{C}^\infty(M)$ is in fact also a $\mathsf{S}^\infty$-derivation. However for a more exotic $\mathcal{C}^\infty$-rings, $\mathsf{S}^\infty$-derivations and $\mathsf{Sym}$-derivations may not necessarily coincide since $\mathsf{S}^\infty$-derivations satisfy a chain rule for smooth functions while $\mathsf{Sym}$-derivations only satisfy a chain rule for polynomial functions\footnote{Admittedly (and embarrassingly), the author does not have a good example where this is the case.}. 

Similarly to the previous examples, since $\mathsf{VEC}_\mathbb{R}$ admits both countable products and countable coproducts, this codifferential category admits both free and cofree $\mathsf{S}^\infty$-differential algebras. As a particular consequence of having cofree $\mathsf{S}^\infty$-differential algebras, this implies that the Hurwitz series ring over a $\mathcal{C}^\infty$-ring is again a $\mathcal{C}^\infty$-ring. 
\end{example}

\begin{example} \normalfont \label{RBex} The integration counterpart of differential algebras are known as Rota-Baxter algebras \cite{guo2012introduction} and surprisingly the free Rota-Baxter algebra construction provides a codifferential category structure \cite[Example 8.7]{blute2018differential}. For a commutative ring $R$, an $R$-Rota-Baxter algebra is a commutative $R$-algebra $A$ equipped with an $R$-linear map $\mathsf{P}: A \to A$ such that $\mathsf{P}$ satisfies the Rota-Baxter identity, that is, the following equality holds:
\[\mathsf{P}(a)\mathsf{P}(b)=\mathsf{P}(a\mathsf{P}(b))+\mathsf{P}(\mathsf{P}(a)b) \quad \forall ab, \in A\]
The Rota-Baxter identity is integration by parts expressed only using integrals. For an $R$-module $M$, there exists a free Rota-Baxter algebra over $M$ \cite[Chapter 3]{guo2012introduction}, $\mathsf{RB}(M)$, and is defined as follows:
\[\mathsf{RB}(M) := \mathsf{Sh}(\mathsf{Sym(M)) \otimes \mathsf{Sym}(M)}\]
where for $\mathsf{Sh}$ denotes the shuffle algebra \cite{guo2012introduction}. This induces an algebra modality $\mathsf{RB}$ on $\mathsf{MOD}_R$ and which comes equipped with deriving transformation defined simply as:
 \[1 \otimes \mathsf{d}: \mathsf{Sh}(\mathsf{Sym}(M)) \otimes \mathsf{Sym}(M) \to \mathsf{Sh}(\mathsf{Sym}(M)) \otimes \mathsf{Sym}(M) \otimes M\]
 where $\mathsf{d}$ is the deriving transformation for $\mathsf{Sym}$. This provides another codifferential category structure on $\mathsf{MOD}_R$. In this case, $\mathsf{RB}$-algebras are precisely $R$-Rota-Baxter algebras. While one might expect $\mathsf{RB}$-differential algebras to be differential Rota-Baxter algebras \cite{guo2008differential} whose derivation and Rota-Baxter operator satisfy the first fundamental theorem of calculus, this is not the case and quite the opposite is true. An $\mathsf{RB}$-differential algebras is an $R$-Rota-Baxter algebra $A$ with Rota-Baxter operator $\mathsf{P}$, which is also an $R$-differential algebra with derivation $\mathsf{D}$, such that $\mathsf{P}$ and $\mathsf{D}$ are incompatible in the sense that $\mathsf{D}(\mathsf{P}(a)) = 0$. For the same reasons discussed in Example \ref{symex}, this codifferential category structure admits both free and cofree $\mathsf{RB}$-differential algebras. 
\end{example}

\section{Free $\mathsf{T}$-Differential Algebras}\label{freesec}

In this section we show for a codifferential category with countable coproducts, its category of $\mathsf{T}$-differential algebras is monadic over the codifferential category. Explicitly, this means that we construct a monad $\mathsf{B}$ whose Eilenberg-Moore category is equivalent to the category of $\mathsf{T}$-differential algebras. The construction of the monad is based on the classical construction of free differential algebras found in \cite[Example 8.3]{blute2018differential}. 

Throughout this section, let $\mathbb{X}$ be a codifferential category with algebra modality $(\mathsf{T}, \mu, \eta, \mathsf{m}, \mathsf{u})$ and deriving transformation $\mathsf{d}$, and countable coproducts $\bigoplus$. 

Define the functor $\mathsf{B}: \mathbb{X} \to \mathbb{X}$ on objects as $\mathsf{B}A:= \mathsf{T}\left(\bigoplus\limits^\infty_{n=0} A\right)$ and on maps as $\mathsf{B}(f) := \mathsf{T}\left(\bigoplus\limits^\infty_{n=0} f\right)$. Now define the natural transformation $\mathsf{b}: \bigoplus\limits^\infty_{n=0} A \to \bigoplus\limits^\infty_{n=0} A$ as the unique map such that for each $n \in \mathbb{N}$ the following diagram commutes (where $\iota_k$ are the injection maps):   
\begin{equation}\label{bdef}\begin{gathered} \xymatrixcolsep{5pc} \xymatrixrowsep{2pc}\xymatrix{A \ar[r]^-{\iota_n} \ar[dr]_-{\iota_{n+1}} & \bigoplus\limits^\infty_{n=0} A\ar[d]^-{\mathsf{\mathsf{b}}} \\
& \bigoplus\limits^\infty_{n=0} A
  } \end{gathered}\end{equation}
\begin{lemma}\label{Tfdifalg} For each object $A$, the triple $(\mathsf{B}(A), \mu, \mathsf{b}^\sharp)$ is a $\mathsf{T}$-differential algebra (where recall $\mathsf{b}^\sharp := \mathsf{d}; (1 \otimes \mathsf{b}); \mathsf{d}^\circ$). Furthermore, for every map $f: A \to A^\prime$, $\mathsf{B}(f): (\mathsf{B}A, \mu, \mathsf{b}^\sharp) \to (\mathsf{B}A^\prime, \mu, \mathsf{b}^\sharp)$ is a $\mathsf{T}$-differential algebra morphism. 
\end{lemma}
\begin{proof} By Corollary \ref{endoTalg}, $(\mathsf{B}(A), \mu, \mathsf{b}^\sharp)$ is a $\mathsf{T}$-differential algebra. To show that $\mathsf{B}(f)$ is a $\mathsf{T}$-differential algebra morphism, notice that $\mathsf{B}(f)$ is already a $\mathsf{T}$-algebra morphism by construction. Therefore, it remains to show that $\mathsf{B}(f)$ commutes with the $\mathsf{T}$-derivations:
\begin{align*}
\mathsf{T}\left(\bigoplus\limits^\infty_{n=0} f\right); \mathsf{b}^\sharp &=~ \mathsf{T}\left(\bigoplus\limits^\infty_{n=0} f\right); \mathsf{d}; (1 \otimes \mathsf{b}); \mathsf{d}^\circ\\
 &=~ \mathsf{d};\left( \mathsf{T}\left(\bigoplus\limits^\infty_{n=0} f\right) \otimes  \bigoplus\limits^\infty_{n=0} f\right); (1 \otimes \mathsf{b}); \mathsf{d}^\circ \tag{Naturality of $\mathsf{d}$} \\
&=~ \mathsf{d}; (1 \otimes \mathsf{b}); \left( \mathsf{T}\left(\bigoplus\limits^\infty_{n=0} f\right) \otimes  \bigoplus\limits^\infty_{n=0} f\right); \mathsf{d}^\circ \tag{Naturality of $\mathsf{b}$} \\
&=~ \mathsf{d}; (1 \otimes \mathsf{b}); \mathsf{d}^\circ; \mathsf{T}\left(\bigoplus\limits^\infty_{n=0} f\right) \tag{Naturality of $\mathsf{d}^\circ$}\\
&=~ \mathsf{b}^\sharp;\mathsf{T}\left(\bigoplus\limits^\infty_{n=0} f\right)
\end{align*}
\end{proof} 

Before defining a monad structure on $\mathsf{B}$, consider first a special map for $\mathsf{T}$-differential algebras in the presence of countable coproducts. For every $\mathsf{T}$-differential algebra $(A, \nu, \mathsf{D})$, define the map $\mathsf{D}^\natural: \bigoplus\limits^\infty_{n=0} A \to A$ as the unique map such that for each $n \in \mathbb{N}$ the following diagram commutes: 
\begin{equation}\label{Dflat0}\begin{gathered} \xymatrixcolsep{5pc}\xymatrixrowsep{4pc}\xymatrix{A \ar[r]^-{\iota_n} \ar[dr]_-{\mathsf{D}^n} & \bigoplus\limits^\infty_{n=0} A\ar[d]^-{\mathsf{D}^\natural} \\
& A
  } \end{gathered}\end{equation}
where recall $\mathsf{D}^0 := 1_A$. The map $\mathsf{D}^\natural$ will be used in the construction of the monad's multiplication. 

\begin{lemma} Let $(A, \nu, \mathsf{D})$ be a $\mathsf{T}$-differential algebra, then the following diagram commutes:
\begin{equation}\label{Dflat1}\begin{gathered} \xymatrixcolsep{5pc}\xymatrix{  \bigoplus\limits^\infty_{n=0} A \ar[d]_-{\mathsf{D}^\natural} \ar[r]^-{\mathsf{b}} & \bigoplus\limits^\infty_{n=0} A \ar[d]^-{\mathsf{D}^\natural} \\
A \ar[r]_-{\mathsf{D}}& A
  } \end{gathered}\end{equation}
  Furthermore, if $f: (A, \nu, \mathsf{D}) \to (B, \nu^\prime, \mathsf{D}^\prime)$ is a $\mathsf{T}$-differential algebra morphism, then following diagram commutes: 
\begin{equation}\label{Dflat2}\begin{gathered} \xymatrixcolsep{5pc}\xymatrix{  \bigoplus\limits^\infty_{n=0} A \ar[d]_-{\mathsf{D}^\natural} \ar[r]^-{\bigoplus\limits^\infty_{n=0} f} & \bigoplus\limits^\infty_{n=0} B \ar[d]^-{(\mathsf{D}^\prime)^\natural} \\
A \ar[r]_-{f}& B
  } \end{gathered}\end{equation}
\end{lemma}
\begin{proof} For the first square, notice that for each $n \in \mathbb{N}$ we have that: 
\[\iota_n; \mathsf{b}; \mathsf{D}^\natural = \iota_{n+1}; \mathsf{D}^\natural = \mathsf{D}^{n+1} =\mathsf{D}^n; \mathsf{D} = \iota_n; \mathsf{D}^\natural; \mathsf{D} \]
Then by the couniversal property of the coproduct, $\mathsf{b}; \mathsf{D}^\natural = \mathsf{D}^\natural; \mathsf{D}$. For the second square, notice that for each $n \in \mathbb{N}$ we have that: 
\[\iota_n; \bigoplus\limits^\infty_{n=0} f; (\mathsf{D}^\prime)^\natural = f; \iota_n; (\mathsf{D}^\prime)^\natural = f; (\mathsf{D}^\prime)^n =  \mathsf{D}^n; f = \iota_n; \mathsf{D}^\natural; f \]
Then by the couniversal property of the coproduct, $\bigoplus\limits^\infty_{n=0} f; (\mathsf{D}^\prime)^\natural = \mathsf{D}^\natural; f$.
 
\end{proof} 

Define the natural transformations $\alpha: A \to \mathsf{B}A$ and $\beta: \mathsf{B}\mathsf{B}A \to \mathsf{B}A$ respectively as follows: 
\begin{equation}\label{}\begin{gathered} \alpha:= \xymatrixcolsep{5pc}\xymatrix{A \ar[r]^-{\iota_0} & \bigoplus\limits^\infty_{n=0} A \ar[r]^-{\eta} & \mathsf{T}\left(\bigoplus\limits^\infty_{n=0} A\right)
  } \end{gathered}\end{equation}
  \begin{equation}\label{}\begin{gathered}  \beta:= \xymatrixcolsep{3pc}\xymatrix{\mathsf{T}\left(\bigoplus\limits^\infty_{n=0} \mathsf{T}\left(\bigoplus\limits^\infty_{n=0} A\right)\right) \ar[r]^-{\mathsf{T}\left( (b^\sharp)^\natural\right)} & \mathsf{T}\mathsf{T}\left(\bigoplus\limits^\infty_{n=0} A\right) \ar[r]^-{\mu} & \mathsf{T}\left(\bigoplus\limits^\infty_{n=0} A\right)   
  }  \end{gathered}\end{equation}

\begin{lemma}\label{betalem} $\beta: (\mathsf{B}\mathsf{B}A, \mu, \mathsf{b}^\sharp) \to (\mathsf{B}(A), \mu, \mathsf{b}^\sharp)$ is a $\mathsf{T}$- differential algebra morphism. 
\end{lemma}
\begin{proof} By construction, $\beta$ is a $\mathsf{T}$-algebra morphism. Therefore, it remains to check that $\beta$ commutes with the $\mathsf{T}$-derivations. 
\begin{align*}
\beta; \mathsf{b}^\sharp &=~ \mathsf{T}\left( (b^\sharp)^\natural\right); \mu; b^\sharp \\
&=~ \mathsf{T}\left( (b^\sharp)^\natural\right); (\mathsf{b}^\sharp)^\sharp; \mu \tag{Corollary \ref{dcircv}}\\
&=~ \mathsf{T}\left( (b^\sharp)^\natural\right);\mathsf{d}; (1 \otimes \mathsf{b}^\sharp); \mathsf{d}^\circ; \mu \\
&=~ \mathsf{d}; \left( \mathsf{T}\left( (b^\sharp)^\natural\right) \otimes (b^\sharp)^\natural \right);  (1 \otimes \mathsf{b}^\sharp); \mathsf{d}^\circ; \mu \tag{Naturality of $\mathsf{d}$}\\
&=~ \mathsf{d}; (1 \otimes \mathsf{b}); \left( \mathsf{T}\left( (b^\sharp)^\natural\right) \otimes (b^\sharp)^\natural \right); \mathsf{d}^\circ; \mu \tag{\ref{Dflat1}}\\
&=~ \mathsf{d}; (1 \otimes \mathsf{b}); \mathsf{d}^\circ; \mathsf{T}\left( (b^\sharp)^\natural\right); \mu \tag{Naturality of $\mathsf{d}^\circ$}\\
&=~ \mathsf{b}^\sharp; \beta
\end{align*}
\end{proof} 
  
\begin{proposition}\label{monadprop} $(\mathsf{B}, \beta, \alpha)$ is a monad. 
\end{proposition}   
\begin{proof} We first show that $\alpha$ and $\beta$ are indeed natural transformations. Naturality of $\alpha$ is straightforward: 
\[f; \alpha = f; \iota_0; \eta = \iota_0; \bigoplus\limits^\infty_{n=0} f; \eta = \iota_0; \eta; \mathsf{T}\left( \bigoplus\limits^\infty_{n=0} f \right) = \alpha; \mathsf{B}(f)  \]
To show the naturality of $\beta$, we use the fact that $\mathsf{B}(f)$ is a $\mathsf{T}$-differential algebra morphism: 
\begin{align*}
\mathsf{B}\mathsf{B}(f); \beta &=~ \mathsf{T}\left(\bigoplus\limits^\infty_{n=0} \mathsf{T}\left(\bigoplus\limits^\infty_{n=0} f\right)\right); \mathsf{T}\left( (b^\sharp)^\natural\right); \mu \\
&=~ \mathsf{T}\left( (b^\sharp)^\natural\right);  \mathsf{T}\mathsf{T}\left(\bigoplus\limits^\infty_{n=0} f\right); \mu \tag{Lemma \ref{Tfdifalg} + (\ref{Dflat2})} \\
&=~ \mathsf{T}\left( (b^\sharp)^\natural\right); \mu; \mathsf{T}\left(\bigoplus\limits^\infty_{n=0} f\right) \tag{Naturality of $\mu$}\\
&=~ \beta; \mathsf{B}(f)
\end{align*}
Now we show that the three monad identities (\ref{monadeq}) hold. 
\begin{enumerate}
\item $\alpha; \beta =1$: This is straightforward: 
\[\alpha; \beta = \iota_0; \eta; \mathsf{T}\left( (b^\sharp)^\natural\right); \mu = \iota_0; (b^\sharp)^\flat; \eta; \mu = (b^\sharp)^0; 1 = 1 \]
\item $\mathsf{B}(\alpha);\beta=1$: First notice that for each $n \in \mathbb{N}$ we have the following equality: 
\begin{align*}
\iota_n; \bigoplus\limits^\infty_{n=0}\iota_0; \bigoplus\limits^\infty_{n=0}\eta; (\mathsf{b}^\sharp)^\natural= \iota_0; \eta; \iota_n; (\mathsf{b}^\sharp)^\natural = \iota_0; \eta; (\mathsf{b}^\sharp)^n = \iota_0; \mathsf{b}^n; \eta= Ê\iota_n; \eta
\end{align*}
So by the couniversal property of the coproduct, $\bigoplus\limits^\infty_{n=0}\iota_0; \bigoplus\limits^\infty_{n=0}\eta; (\mathsf{b}^\sharp)^\natural = \eta$. From this equality, it then follows that: 
\[\mathsf{B}(\alpha);\beta = \mathsf{T}\left(\bigoplus\limits^\infty_{n=0} \iota_0 \right); \mathsf{T}\left(\bigoplus\limits^\infty_{n=0} \eta \right); \mathsf{T}\left( (b^\sharp)^\natural\right); \mu = \mathsf{T}(\eta);\mu = 1 \]
\item $\mathsf{B}(\beta); \beta = \beta; \beta$: Here we use the fact that $\beta$ is a $\mathsf{T}$-differential algebra morphism. 
\begin{align*}
\mathsf{B}(\beta); \beta &=~ \mathsf{T}\left(\bigoplus\limits^\infty_{n=0} \beta \right); \mathsf{T}\left( (b^\sharp)^\natural\right); \mu \\
&=~ \mathsf{T}\left( (b^\sharp)^\natural\right); \mathsf{T}(\beta); \mu \tag{Lemma \ref{betalem} + (\ref{Dflat2})} \\
&=~ \mathsf{T}\left( (b^\sharp)^\natural\right); \mathsf{T}\mathsf{T}\left( (b^\sharp)^\natural\right); \mathsf{T}(\mu); \mu \\
&=~ \mathsf{T}\left( (b^\sharp)^\natural\right); \mathsf{T}\mathsf{T}\left( (b^\sharp)^\natural\right); \mu; \mu \tag{\ref{monadeq}} \\
&=~ \mathsf{T}\left( (b^\sharp)^\natural\right); \mu; \mathsf{T}\left( (b^\sharp)^\natural\right); \mu \tag{Naturality of $\mu$} \\
&=~ \beta;\beta 
\end{align*} 
\end{enumerate}
\end{proof} 

We now wish to show that the Eilenberg-Moore category of $\mathsf{B}$ is equivalent to category of $\mathsf{T}$-differential algebras. To do so, we will show that $\mathsf{B}$-algebras are precisely the $\mathsf{T}$-differential algebras and that $\mathsf{B}$-algebra morphisms are precisely $\mathsf{T}$-differential algebra morphisms. 

\begin{lemma}\label{BalgDalg} Let $(A, \chi)$ be a $\mathsf{B}$-algebra. Define the maps $\nu^\chi: \mathsf{T}A \to A$ and $\mathsf{D}^\chi: A \to A$ respectively as follows: 
\begin{equation}\label{}\begin{gathered} \nu^\chi := \xymatrix{\mathsf{T}A \ar[r]^-{\mathsf{T}(\iota_0)} & \mathsf{T}\left(\bigoplus\limits^\infty_{n=0} A\right) \ar[r]^-{\chi} & A
  } \quad \quad \mathsf{D}^\chi := \xymatrix{A \ar[r]^-{\eta} & \mathsf{T}A \ar[r]^-{\mathsf{T}(\iota_1)} & \mathsf{T}\left(\bigoplus\limits^\infty_{n=0} A\right) \ar[r]^-{\chi} & A
  } \end{gathered}\end{equation}
  Then $(A, \nu^\chi, \mathsf{D}^\chi)$ is a $\mathsf{T}$-differential algebra. Furthermore, if $f: (A, \chi) \to (A^\prime, \chi^\prime)$ is a $\mathsf{B}$-algebra morphism, then $f: (A, \nu^\chi, \mathsf{D}^\chi) \to (A^\prime, \nu^{\chi^\prime}, \mathsf{D}^{\chi^\prime})$ is a $\mathsf{T}$-differential algebra morphism. 
\end{lemma}
\begin{proof} We first show that $\nu^\chi$ satisfies the $\mathsf{T}$-algebra identities (\ref{Talgdef}):
\begin{align*}
\eta; \nu^\chi = \eta; \mathsf{T}(\iota_0); \chi = \iota_0; \eta; \chi  = \alpha; \chi = 1 
\end{align*}
\begin{align*}
\mathsf{T}(\nu^\chi); \nu^\chi &=~\mathsf{T}\mathsf{T}(\iota_0); \mathsf{T}(\chi); \mathsf{T}(\iota_0); \chi \\
&=~ \mathsf{T}\mathsf{T}(\iota_0); \mathsf{T}(\iota_0); \mathsf{T}\left(\bigoplus\limits^\infty_{n=0} \chi \right);   \chi \\
&=~ \mathsf{T}\mathsf{T}(\iota_0); \mathsf{T}(\iota_0); \mathsf{B}(\chi) ;   \chi \\
&=~ \mathsf{T}\mathsf{T}(\iota_0); \mathsf{T}(\iota_0); \beta ;   \chi \tag{\ref{Talgdef}} \\
&=~ \mathsf{T}\mathsf{T}(\iota_0); \mathsf{T}(\iota_0); \mathsf{T}\left( (b^\sharp)^\natural\right); \mu; \chi \\
&=~ \mathsf{T}\mathsf{T}(\iota_0); \mu; \chi \tag{\ref{Dflat1}} \\
&=~ \mu; \mathsf{T}(\iota_0); \chi \tag{Naturality of $\mu$} \\
&=~  \mu; \nu^\chi
\end{align*}
Next we show that $\mathsf{D}^\chi$ is a $\mathsf{T}$-derivation. First observe that $\chi$ preserves the multiplication in the following sense: 
\begin{align*}
\mathsf{m}; \chi &=~ (\alpha \otimes \alpha); (\beta \otimes \beta); \mathsf{m}; \chi \tag{\ref{Talgdef}} \\
&=~ (\alpha \otimes \alpha); \mathsf{m}; \beta; \chi \tag{$\beta$ is a monoid morphism} \\
&=~ (\alpha \otimes \alpha); \mathsf{m}; \mathsf{B}(\chi); \chi \tag{\ref{Talgdef}} \\
&=~ (\alpha \otimes \alpha); (\mathsf{B}(\chi) \otimes \mathsf{B}(\chi)); \mathsf{m}; \chi \tag{$\mathsf{B}(\chi)$ is a monoid morphism} \\
&=~Ê(\chi \otimes \chi); (\alpha \otimes \alpha); \mathsf{m}; \chi \tag{Naturality of $\alpha$} \\
&=~ (\chi \otimes \chi); (\eta \otimes \eta); (\mathsf{T}(\iota_0) \otimes \mathsf{T}(\iota_0)); \mathsf{m}; \chi \\
&=~ (\chi \otimes \chi); (\eta \otimes \eta); \mathsf{m}; \mathsf{T}(\iota_0); \chi \tag{Naturality of $\mathsf{m}$}\\
&=~ (\chi \otimes \chi);  (\eta \otimes \eta); \mathsf{m}; \nu^\chi \\
&=~ (\chi \otimes \chi); \mathsf{m}^{\nu^\chi}
\end{align*}
Now we show that $\mathsf{D}^\chi$ satisfies (\ref{Tdiffalg}): 
\begin{align*}
\nu^\chi; \mathsf{D}^\chi &=~ \mathsf{T}(\iota_0); \chi; \eta; \mathsf{T}(\iota_1); \chi \\
&=~ \mathsf{T}(\iota_0); \eta; \mathsf{T}(\iota_1); \mathsf{T}\left(\bigoplus\limits^\infty_{n=0} \chi \right); \chi \tag{Naturality of $\eta$ and $\iota_1$} \\
&=~Ê\mathsf{T}(\iota_0); \eta; \mathsf{T}(\iota_1); \mathsf{B}(\chi); \chi \\
&=~Ê\mathsf{T}(\iota_0); \eta; \mathsf{T}(\iota_1); \beta; \chi \tag{\ref{Talgdef}} \\
&=~ \mathsf{T}(\iota_0); \eta; \mathsf{T}(\iota_1);  \mathsf{T}\left( (\mathsf{b}^\sharp)^\natural\right); \mu; \chi \\
&=~Ê\mathsf{T}(\iota_0); \eta; \mathsf{T}(\mathsf{b}^\sharp); \mu; \chi \tag{\ref{Dflat0}} \\
&=~Ê\mathsf{T}(\iota_0); \mathsf{b}^\sharp; \eta; \mu; \chi \tag{Naturality of $\eta$} \\
&=~Ê\mathsf{T}(\iota_0); \mathsf{b}^\sharp; \chi \tag{\ref{Talgdef}} \\
&=~Ê\mathsf{T}(\iota_0); \mathsf{d}; (1 \otimes \mathsf{b}); \mathsf{d}^\circ; \chi \\
&=~Ê\mathsf{d}; (\mathsf{T}(\iota_0) \otimes \iota_0); (1 \otimes \mathsf{b}); \mathsf{d}^\circ; \chi \\
&=~ Ê\mathsf{d}; (\mathsf{T}(\iota_0) \otimes \iota_1); \mathsf{d}^\circ; \chi \tag{\ref{bdef}} \\
&=~ \mathsf{d}; (\mathsf{T}(\iota_0) \otimes \iota_1); (1 \otimes \eta); \mathsf{m}; \chi \\
&=~ \mathsf{d}; (\mathsf{T}(\iota_0) \otimes \iota_1); (1 \otimes \eta); (\chi \otimes \chi); \mathsf{m}^{\nu^\chi} \tag{$\chi$ preserves the multiplication} \\
&=~ \mathsf{d}; (\nu^\chi  \otimes \mathsf{D}^\chi); \mathsf{m}^{\nu^\chi}
\end{align*}  
Therefore $(A, \nu^\chi, \mathsf{D}^\chi)$ is a $\mathsf{T}$-differential algebra. Now suppose that $f$ is a $\mathsf{B}$-algebra morphism. We first show that $f$ is a $\mathsf{T}$-algebra morphism:
\[\mathsf{T}(f);  \nu^{\chi^\prime} = \mathsf{T}(f); \mathsf{T}(\iota_0); \chi^\prime= \mathsf{T}(\iota_0); \mathsf{T}\left(\bigoplus\limits^\infty_{n=0} f \right) ; \chi^\prime = \mathsf{T}(\iota_0); \mathsf{B}(f);  \chi^\prime= \mathsf{T}(\iota_0); \chi  ; f = \nu^\chi; f  \]
And finally we show that $f$ commutes with the $\mathsf{T}$-derivations: 
\begin{align*}
f; \mathsf{D}^{\chi^\prime} &=~ f; \eta; \mathsf{T}(\iota_1); \chi^\prime =Ê\eta; \mathsf{T}(\iota_1); \mathsf{T}\left(\bigoplus\limits^\infty_{n=0} f \right) ; \chi^\prime =Ê\eta; \mathsf{T}(\iota_1); \mathsf{B}(f);  \chi^\prime = \eta; \mathsf{T}(\iota_1); \chi  ; f = \mathsf{D}^\chi; f
\end{align*}
\end{proof} 

In order to prove the converse of Lemma \ref{BalgDalg}, we will require the following observation: 

\begin{lemma}\label{DBlem} If $(A, \nu, \mathsf{D})$ is a $\mathsf{T}$-differential algebra, then $\mathsf{T}(\mathsf{D}^\natural): (\mathsf{B}A, \mu, \mathsf{b}^\sharp) \to (\mathsf{T}A, \mu, \mathsf{D}^\sharp)$ is a $\mathsf{T}$-differential algebra morphism. 
\end{lemma}
\begin{proof} By construction, $\mathsf{T}(\mathsf{D}^\natural)$ is a $\mathsf{T}$-algebra morphism. Therefore, we must simply check that $\mathsf{T}(\mathsf{D}^\natural)$ commutes with the $\mathsf{T}$-derivations:
\begin{align*}
\mathsf{b}^\sharp; \mathsf{T}(\mathsf{D}^\natural)&=~ \mathsf{d}; (1 \otimes \mathsf{b}); \mathsf{d}^\circ; \mathsf{T}(\mathsf{D}^\natural) \\
&=~  \mathsf{d}; (1 \otimes \mathsf{b}); (\mathsf{T}(\mathsf{D}^\natural) \otimes \mathsf{D}^\natural); \mathsf{d}^\circ \tag{Naturality of $\mathsf{d}^\circ$} \\
&=~ \mathsf{d}; (\mathsf{T}(\mathsf{D}^\natural) \otimes \mathsf{D}^\natural); (1 \otimes \mathsf{D});\mathsf{d}^\circ \tag{\ref{Dflat1}}\\
&=~ \mathsf{T}(\mathsf{D}^\natural); \mathsf{d}; (1 \otimes \mathsf{b}); \mathsf{d}^\circ  \tag{Naturality of $\mathsf{d}$}\\
&=~ \mathsf{T}(\mathsf{D}^\natural); \mathsf{D}^\sharp
\end{align*}
\end{proof} 

\begin{lemma}\label{DalgBalg} Let $(A, \nu, \mathsf{D})$ be a $\mathsf{T}$-differential algebra. Define the map $\chi^{(\nu,\mathsf{D})}: \mathsf{B}A \to A$ as follows:
\begin{equation}\label{}\begin{gathered} \xymatrixcolsep{5pc}\xymatrix{  \mathsf{T}\left(\bigoplus\limits^\infty_{n=0} A\right) \ar[r]^-{\mathsf{T}(\mathsf{D}^\natural)} & \mathsf{T}(A) \ar[r]^-{\nu} & A
  } \end{gathered}\end{equation}
  Then $(A, \chi^{(\nu,\mathsf{D})})$ is a $\mathsf{B}$-algebra. Furthermore, if $f: (A, \nu, \mathsf{D}) \to (A^\prime, \nu^\prime, \mathsf{D}^\prime)$ is a $\mathsf{T}$-differential algebra morphism, then $f: (A, \chi^{(\nu,\mathsf{D})}) \to (A^\prime, \chi^{(\nu^\prime,\mathsf{D}^\prime)})$ is a $\mathsf{B}$-algebra morphism. 
\end{lemma}
\begin{proof} We first show that $\chi^{(\nu,\mathsf{D})}$ satisfies the $\mathsf{B}$-algebra identities (\ref{Talgdef}): 
\begin{align*}
\alpha; \chi^{(\nu,\mathsf{D})} = \eta; \mathsf{T}(\iota_0); \mathsf{T}(\mathsf{D}^\natural); \nu = \eta; \nu =1 
\end{align*}
\begin{align*}
\mathsf{B}(\chi^{(\nu,\mathsf{D})}); \chi^{(\nu,\mathsf{D})} &=~Ê\mathsf{T}\left(\bigoplus\limits^\infty_{n=0} \mathsf{T}(\mathsf{D}^\natural) \right); \mathsf{T}\left(\bigoplus\limits^\infty_{n=0} \nu \right); \mathsf{T}(\mathsf{D}^\natural); \nu \\
&=~ \mathsf{T}\left(\bigoplus\limits^\infty_{n=0} \mathsf{T}(\mathsf{D}^\natural) \right); \mathsf{T}((\mathsf{D}^\sharp)^\natural); \mathsf{T}(\nu); \nu \tag{Corollary \ref{dcircv} + (\ref{Dflat2})} \\
&=~ \mathsf{T}\left( (b^\sharp)^\natural\right); \mathsf{T}\mathsf{T}(\mathsf{D}^\natural); \mathsf{T}(\nu); \nu \tag{Lemma \ref{DBlem} + (\ref{Dflat2})} \\
&=~ \mathsf{T}\left( (b^\sharp)^\natural\right); \mathsf{T}\mathsf{T}(\mathsf{D}^\natural); \mu; \nu \tag{\ref{Talgdef}} \\
&=~Ê\mathsf{T}\left( (b^\sharp)^\natural\right); \mu; \mathsf{T}(\mathsf{D}^\natural); \nu \tag{Naturality of $\mu$} \\
&=~ \beta; \chi^{(\nu,\mathsf{D})}
\end{align*}
Next, suppose that $f$ is a $\mathsf{T}$-differential algebra morphism. Then by (\ref{Dflat2}), we have that: 
\begin{align*}
\mathsf{B}(f); \chi^{(\nu^\prime,\mathsf{D}^\prime)} = \mathsf{T}\left(\bigoplus\limits^\infty_{n=0} f\right); \mathsf{T}((\mathsf{D}^\prime)^\natural); \nu^\prime = \mathsf{T}(\mathsf{D}^\natural); \mathsf{T}(f); \nu= \mathsf{T}(\mathsf{D}^\natural); \nu; f =  \chi^{(\nu,\mathsf{D})}; f
\end{align*}
\end{proof} 

\begin{theorem}\label{freethm} The Eilenberg-Moore category of the monad $(\mathsf{B}, \beta, \alpha)$ is isomorphic to category of $\mathsf{T}$-differential algebras of $\mathbb{X}$.  
\end{theorem} 
\begin{proof} To prove this theorem, it suffices to show that the constructions of Lemma \ref{BalgDalg} and Lemma \ref{DalgBalg} are inverses of each other. Starting with a $\mathsf{B}$-algebra $(A, \chi)$, we must show that $\chi^{(\nu^\chi, \mathsf{D}^\chi)}= \chi$. Recall that in the proof of Proposition \ref{monadprop}, we showed that $\bigoplus\limits^\infty_{n=0} \alpha; (\mathsf{b}^\sharp)^\natural = \eta$. Then by using this identity, and the fact $\chi$ is a $\mathsf{B}$-algebra morphism and therefore also a $\mathsf{T}$-differential algebra morphism, we have that: 
\begin{align*}
\chi^{(\nu^\chi, \mathsf{D}^\chi)} &=~ \mathsf{B}(\alpha); \mathsf{B}(\chi); \chi^{(\nu^\chi, \mathsf{D}^\chi)} \tag{\ref{Talgdef}} \\
&=~ \mathsf{B}(\alpha); \mathsf{T}\left(\bigoplus\limits^\infty_{n=0} \chi \right);  \mathsf{T}((\mathsf{D}^\chi)^\natural); \nu^\chi \\
&=~ \mathsf{B}(\alpha);  \mathsf{T}((\mathsf{b}^\sharp)^\natural); \mathsf{T}(\chi); \nu^\chi \tag{Lemma \ref{BalgDalg} + (\ref{Dflat2})} \\
&=~Ê\mathsf{T}\left(\bigoplus\limits^\infty_{n=0} \alpha \right); \mathsf{T}((\mathsf{b}^\sharp)^\natural); \mathsf{T}(\chi); \nu^\chi \\
&=~ \mathsf{T}(\eta); \mathsf{T}(\chi); \nu^\chi \\
&=~Ê\mathsf{T}(\eta); \mathsf{T}(\chi); \mathsf{T}(\iota_0); \chi \\
&=~Ê \mathsf{T}(\eta); \mathsf{T}(\iota_0); Ê\mathsf{T}\left(\bigoplus\limits^\infty_{n=0} \chi \right); \chi \\
&=~ \mathsf{T}(\eta); \mathsf{T}(\iota_0); \mathsf{B}(\chi); \chi \\
&=~ \mathsf{T}(\eta); \mathsf{T}(\iota_0); \beta; \chi \tag{\ref{Talgdef}} \\
&=~Ê\mathsf{T}(\eta); \mathsf{T}(\iota_0);  \mathsf{T}\left( (b^\sharp)^\natural\right); \mu; \chi \\
&=~ \mathsf{T}(\eta); \mu; \chi \\
&=~ \chi \tag{\ref{Talgdef}}
\end{align*}
Conversly, given a $\mathsf{T}$-differential algebra $(A, \nu, \mathsf{D})$, we must show that $\nu^{\chi^{(\nu,\mathsf{D})}} = \nu$ and $\mathsf{D}^{\chi^{(\nu,\mathsf{D})}}= \mathsf{D}$. 
\begin{align*}
\nu^{\chi^{(\nu,\mathsf{D})}} = \mathsf{T}(\iota_0) ;\chi^{(\nu,\mathsf{D})} = \mathsf{T}(\iota_0); \mathsf{T}(\mathsf{D}^\natural); \nu = \nu
\end{align*}
\begin{align*}
\mathsf{D}^{\chi^{(\nu,\mathsf{D})}} = \eta; \mathsf{T}(\iota_1); \chi^{(\nu,\mathsf{D})} = \eta; \mathsf{T}(\iota_1);\mathsf{T}(\mathsf{D}^\natural); \nu = \eta; \mathsf{T}(\mathsf{D}); \nu= \mathsf{D}; \eta; \nu = \mathsf{D}
\end{align*}
\end{proof}

As a consequence of Theorem \ref{freethm}, for each object $A$, $(\mathsf{B}(A), \mu, \mathsf{b}^\sharp)$ is the free $\mathsf{T}$-differential algebra over $A$. 

\section{Cofree $\mathsf{T}$-Differential Algebras and Power Series Over $\mathsf{T}$-Algebras}\label{cofreesec}

In this section we show that for a codifferential category with countable products, the category of $\mathsf{T}$-differential algebras is comonadic over the category of $\mathsf{T}$-algebras. Explicitly, we construct a comonad on the category of $\mathsf{T}$-algebras whose co-Eilenberg-Moore category is equivalent to the category of $\mathsf{T}$-differential algebras. As previously advertised in Example \ref{cofreediff}, this construction is generalization of Hurwitz series rings \cite{keigher1997ring} which are themselves cofree differential algebras \cite{guo2008differential, keigher1975adjunctions}. This generalized construction implies that Hurwitz series algebras inherit more then just algebra structure. For example, as mentioned in Example \ref{smoothex}, Hurwitz series ring over $\mathcal{C}^\infty$-rings are again $\mathcal{C}^\infty$-rings (which to the author's knowledge is a new observation). Towards the end of this section we also discuss power series over $\mathsf{T}$-algebras and show their relation to cofree $\mathsf{T}$-differential algebras. 

Throughout this section, let $\mathbb{X}$ be a codifferential category with algebra modality $(\mathsf{T}, \mu, \eta, \mathsf{m}, \mathsf{u})$, deriving transformation $\mathsf{d}$, and countable products $\prod$. Classically, for an $R$-algebra $A$, the underlying set of the Hurwitz series ring over $A$ is isomorphic to the set $\prod \limits^\infty_{n=0} A$. As we will be working in a codifferential category with countable products, our first objective will be, for a $\mathsf{T}$-algebra $(A, \nu)$, to define a $\mathsf{T}$-differential algebra structure on $\prod \limits^\infty_{n=0} A$ such that the underlying differential algebra structure is precisely that of Hurwitz series rings (Example \ref{cofreediff}).

For a $\mathsf{T}$-algebra $(A, \nu)$, define for each $n \in \mathbb{N}$ the maps $\omega^\nu_n: \mathsf{T}\left( \prod \limits^\infty_{n=0} A \right) \to A$ inductively as: 
\begin{enumerate}[{\em (i)}]
\item $\omega^\nu_0 := \mathsf{T}(\pi_0); \nu$
\item $\omega^\nu_{n+1} := \sum\limits^{n}_{k=0} \binom{n}{k} \cdot \mathsf{d}; (\omega^\nu_k \otimes \pi_{n-k+1}); \mathsf{m}^\nu$
\end{enumerate}
where $\pi_j: \prod \limits^\infty_{n=0} A \to A$ is the $j$-th projection. And define the map $\omega^\nu: \mathsf{T}\left( \prod \limits^\infty_{n=0} A \right) \to \prod \limits^\infty_{n=0} A$ as $\omega^\nu := \langle \omega^\nu_n \vert ~ n \in \mathbb{N} \rangle$, that is, the unique map such that for each $n \in \mathbb{N}$, the following diagram commutes: 
\begin{equation}\label{omegadef}\begin{gathered} \xymatrixcolsep{5pc}\xymatrixrowsep{1pc}\xymatrix{\mathsf{T}\left( \prod \limits^\infty_{n=0} A \right)  \ar[r]^-{\omega^\nu} \ar[ddr]_-{\omega^\nu_{n}} & \prod \limits^\infty_{n=0} A \ar[dd]^-{\pi_n} \\ \\
& A
  } \end{gathered}\end{equation}
We will soon show that $(\prod \limits^\infty_{n=0} A, \omega^\nu)$ is in fact a $\mathsf{T}$-algebra but it is important to note that $\omega^\nu$ is \emph{NOT} the canonical $\mathsf{T}$-algebra of $\prod \limits^\infty_{n=0} A$.  Instead, $\omega^\nu$ is the necessary $\mathsf{T}$-algebra structure to obtain a monoid structure analogous to that of the Hurwitz series ring (\ref{hurwtizproduct}). 
\begin{lemma}\label{omegalemma} For a $\mathsf{T}$-algebra $(A, \nu)$, the following equalities hold (for all $n \in \mathbb{N}$): 
\begin{enumerate}[{\em (i)}]
\item $\mathsf{u} ; \omega^\nu_n = \begin{cases} \mathsf{u}^\nu & \text{if } n=0 \\
0 & \text{o.w.} 
\end{cases}$
\item $\eta; \omega^\nu_n = \pi_n$
\item $\mathsf{m}; \omega^\nu_n = \sum\limits^{n}_{k=0} \binom{n}{k} \cdot (\omega^\nu_k \otimes \omega^\nu_{n-k}); \mathsf{m}^\nu$
\item $\mu; \omega^\nu_n = \mathsf{T}(\omega^\nu); \omega^\nu_n$
\end{enumerate}
\end{lemma}
\begin{proof} $(i)$: Since $\omega^\nu_0$ is the composite of monoid morphisms, it follows that $\omega^\nu_0$ is itself a monoid morphism and therefore preserves the unit, i.e.,  $\mathsf{u};\omega^\nu_0=\mathsf{u}^\nu$. For $\omega^\nu_{n+1}$, $n \in \mathbb{N}$, we have that: 
\begin{align*}
\mathsf{u}; \omega^\nu_{n+1} &=~Ê\mathsf{u}; \left( \sum\limits^{n}_{k=0} \binom{n}{k} \cdot   \mathsf{d}; (\omega^\nu_k \otimes \pi_{n-k+1}); \mathsf{m}^\nu \right) \\
&=~\sum\limits^{n}_{k=0} \binom{n}{k} \cdot  \mathsf{u};\mathsf{d}; (\omega^\nu_k \otimes \pi_{n-k+1}); \mathsf{m}^\nu  \\
&=~\sum\limits^{n}_{k=0} \binom{n}{k} \cdot  0; (\omega^\nu_k \otimes \pi_{n-k+1}); \mathsf{m}^\nu \tag{Constant rule \textbf{[d.1]}} \\
&=~ 0
\end{align*}
$(ii)$: For $n=0$ we have that: 
\begin{align*}
\eta; \omega^\nu_0 =Ê\eta; \mathsf{T}(\pi_0); \nu =Ê\pi_0; \eta; \nu =Ê\pi_0 
\end{align*}
While for all $n\in \mathbb{N}$, using $(i)$ we have that: 
\begin{align*}
\eta; \omega^\nu_{n+1} &=~Ê\eta; \left( \sum\limits^{n}_{k=0} \binom{n}{k} \cdot   \mathsf{d}; (\omega^\nu_k \otimes \pi_{n-k+1}); \mathsf{m}^\nu \right) \\
&=~\sum\limits^{n}_{k=0} \binom{n}{k} \cdot  \eta;\mathsf{d}; (\omega^\nu_k \otimes \pi_{n-k+1}); \mathsf{m}^\nu  \\
&=~\sum\limits^{n}_{k=0} \binom{n}{k} \cdot  (\mathsf{u} \otimes 1);(\omega^\nu_k \otimes \pi_{n-k+1}); \mathsf{m}^\nu \tag{Linear rule \textbf{[d.3]}}  \\
&=~ (\mathsf{u} \otimes 1);(\omega^\nu_0 \otimes \pi_{n+1}); \mathsf{m}^\nu +  \sum\limits^{n}_{k=1} \binom{n}{k} \cdot  (\mathsf{u} \otimes 1);(\omega^\nu_k \otimes \pi_{n-k+1}); \mathsf{m}^\nu \tag{Extracting $k=0$}  \\
&=~ (\mathsf{u}^\nu \otimes \pi_{n+1}); \mathsf{m}^\nu +  \sum\limits^{n}_{k=1} \binom{n}{k} \cdot  (0 \otimes \pi_{n-k+1}); \mathsf{m}^\nu \tag{Lemma \ref{omegalemma} (i)}  \\
&=~  \pi_{n+1}; (\mathsf{u}^\nu \otimes 1); \mathsf{m}^\nu + 0  \\
&=~Ê\pi_{n+1}
\end{align*}
$(iii)$: We prove this by induction on $n$. For the base case $n=0$, as explained in $(i)$ above, $\omega^\nu_0$ is a monoid morphism and therefore preserves the multiplication, i.e., $\mathsf{m};\omega^\nu_0= (\omega^\nu_0 \otimes \omega^\nu_0);\mathsf{m}^\nu$. Now suppose the identity holds for $k \leq n$, we now prove it holds for $n+1$:
\begin{align*}
&\mathsf{m}; \omega^\nu_{n+1} =~ \mathsf{m}; \left( \sum\limits^{n}_{k=0} \binom{n}{k} \cdot   \mathsf{d}; (\omega^\nu_k \otimes \pi_{n-k+1}); \mathsf{m}^\nu \right) \\
&=~\sum\limits^{n}_{k=0} \binom{n}{k} \cdot  \mathsf{m};\mathsf{d}; (\omega^\nu_k \otimes \pi_{n-k+1}); \mathsf{m}^\nu  \\
&=~Ê\sum\limits^{n}_{k=0} \binom{n}{k} \cdot \left( (\mathsf{d} \otimes 1);(1 \otimes \sigma);(\mathsf{m} \otimes 1) + (1 \otimes \mathsf{d});(\mathsf{m} \otimes 1) \right) ;(\omega^\nu_k \otimes \pi_{n-k+1}); \mathsf{m}^\nu \tag{Leibniz rule \textbf{[d.2]}}   \\
&=~Ê\sum\limits^{n}_{k=0} \binom{n}{k} \cdot  (\mathsf{d} \otimes 1);(1 \otimes \sigma);(\mathsf{m} \otimes 1);(\omega^\nu_k \otimes \pi_{n-k+1}); \mathsf{m}^\nu \\
&~~+Ê\sum\limits^{n}_{k=0} \binom{n}{k} \cdot (1 \otimes \mathsf{d});(\mathsf{m} \otimes 1); (\omega^\nu_k \otimes \pi_{n-k+1}); \mathsf{m}^\nu  \\
&=~Ê\sum\limits^{n}_{k=0} \binom{n}{k} \cdot  (\mathsf{d} \otimes 1);(1 \otimes \sigma);(\left(\sum\limits^{k}_{j=0} \binom{k}{j} \cdot (\omega^\nu_j \otimes \omega^\nu_{k-j}); \mathsf{m}^\nu \right) \otimes \pi_{n-k+1}); \mathsf{m}^\nu  \\
&~~+\sum\limits^{n}_{k=0} \binom{n}{k} \cdot  (1 \otimes \mathsf{d});(\left(\sum\limits^{k}_{j=0} \binom{k}{j} \cdot (\omega^\nu_j \otimes \omega^\nu_{k-j}); \mathsf{m}^\nu \right) \otimes \pi_{n-k+1}); \mathsf{m}^\nu \tag{Induction hypothesis} \\
&=~Ê\sum\limits^{n}_{k=0} \binom{n}{k} \cdot \sum\limits^{k}_{j=0} \binom{k}{j} \cdot (\mathsf{d} \otimes 1);(1 \otimes \sigma);(\omega^\nu_j \otimes \omega^\nu_{k-j} \otimes \pi_{n-k+1});(\mathsf{m}^\nu \otimes 1); \mathsf{m}^\nu  \\
&~~+ \sum\limits^{n}_{k=0} \binom{n}{k} \cdot \sum\limits^{k}_{j=0} \binom{k}{j} \cdot(1 \otimes \mathsf{d});(\omega^\nu_j \otimes \omega^\nu_{k-j} \otimes \pi_{n-k+1});(\mathsf{m}^\nu \otimes 1); \mathsf{m}^\nu  \\
&=~Ê\sum\limits^{n}_{k=0} \binom{n}{k} \cdot \sum\limits^{k}_{j=0} \binom{k}{j} \cdot (\mathsf{d} \otimes 1);(\omega^\nu_j \otimes \pi_{n-k+1} \otimes \omega^\nu_{k-j});(\mathsf{m}^\nu \otimes 1); \mathsf{m}^\nu  \\
&~~+\sum\limits^{n}_{k=0} \binom{n}{k} \cdot \sum\limits^{k}_{j=0} \binom{k}{j} \cdot(1 \otimes \mathsf{d});(\omega^\nu_j \otimes \omega^\nu_{k-j} \otimes \pi_{n-k+1});(1 \otimes \mathsf{m}^\nu); \mathsf{m}^\nu \tag{Assoc. \& Comm. of Mult.}\\
&=~Ê\sum\limits^{n}_{k=0} \binom{n}{k} \cdot  (\mathsf{d} \otimes 1);(\omega^\nu_k \otimes \pi_{n-k+1} \otimes \omega^\nu_{0});(\mathsf{m}^\nu \otimes 1); \mathsf{m}^\nu \\
&~~+ \sum\limits^{n}_{k=1} \binom{n}{k}  \cdot \sum\limits^{k-1}_{j=0} \binom{k-1}{j} \cdot (\mathsf{d} \otimes 1); (\omega^\nu_j \otimes \pi_{k-1-j+1} \otimes \omega^\nu_{n+1-k}); (\mathsf{m}^\nu \otimes 1); \mathsf{m}^\nu \\
&~~+ \sum\limits^{n}_{k=0} \binom{n}{k} \cdot  (1 \otimes \mathsf{d});(\omega^\nu_0 \otimes \omega^\nu_{k} \otimes \pi_{n-k+1});(1 \otimes \mathsf{m}^\nu); \mathsf{m}^\nu  \\
&~~+ \sum\limits^{n}_{k=1}  \binom{n}{k-1}  \cdot  \sum\limits^{n-k}_{j=0} \binom{n-k}{j} \cdot (1 \otimes \mathsf{d}); (\omega^\nu_k \otimes \omega^\nu_j \otimes \pi_{n-k-j+1}); (1 \otimes \mathsf{m}^\nu); \mathsf{m}^\nu \tag{Reindexing}\\
&=~ (\left(\sum\limits^{n}_{k=0} \binom{n}{k} \cdot \mathsf{d}; (\omega^\nu_k \otimes \pi_{n-k+1}); \mathsf{m}^\nu\right) \otimes \omega^\nu_{0}); \mathsf{m}^\nu\\
&~~+ \sum\limits^{n}_{k=1} \binom{n}{k}  \cdot (\left( \sum\limits^{k-1}_{j=0} \binom{k-1}{j} \cdot \mathsf{d}; (\omega^\nu_j \otimes \pi_{k-1-j+1}); \mathsf{m}^\nu \right) \otimes \omega^\nu_{n+1-k}); \mathsf{m}^\nu \\
&~~+ (\omega^\nu_{0} \otimes \left(\sum\limits^{n}_{k=0} \binom{n}{k} \cdot \mathsf{d}; (\omega^\nu_k \otimes \pi_{n-k+1}); \mathsf{m}^\nu\right)); \mathsf{m}^\nu \\  
&~~+ \sum\limits^{n}_{k=1}  \binom{n}{k-1}  \cdot (\omega^\nu_k \otimes \left( \sum\limits^{n-k}_{j=0} \binom{n-k}{j} \cdot \mathsf{d}; (\omega^\nu_j \otimes \pi_{n-k-j+1}); \mathsf{m}^\nu \right)); \mathsf{m}^\nu\\
&=~Ê(\omega^\nu_0 \otimes \omega^\nu_{n+1});\mathsf{m}^\nu + \sum\limits^{n}_{k=1} \binom{n}{k}  \cdot (\omega^\nu_k \otimes \omega^\nu_{n+1-k}); \mathsf{m}^\nu \\
&~~~+ (\omega^\nu_{n+1} \otimes \omega^\nu_0);\mathsf{m}^\nu + \sum\limits^{n}_{k=1}  \binom{n}{k-1}  \cdot (\omega^\nu_k \otimes \omega^\nu_{n+1-k}); \mathsf{m}^\nu\\
&=~Ê(\omega^\nu_0 \otimes \omega^\nu_{n+1});\mathsf{m}^\nu + (\omega^\nu_{n+1} \otimes \omega^\nu_0);\mathsf{m}^\nu + \sum\limits^{n}_{k=1} \left(\binom{n}{k} + \binom{n}{k-1} \right)  \cdot (\omega^\nu_k \otimes \omega^\nu_{n+1-k}); \mathsf{m}^\nu\\
&=~Ê(\omega^\nu_0 \otimes \omega^\nu_{n+1});\mathsf{m}^\nu + (\omega^\nu_{n+1} \otimes \omega^\nu_0);\mathsf{m}^\nu + \sum\limits^{n}_{k=1} \binom{n+1}{k} \cdot (\omega^\nu_k \otimes \omega^\nu_{n+1-k}); \mathsf{m}^\nu \tag{\ref{binom}}\\
&=~Ê\sum\limits^{n+1}_{k=0} \binom{n+1}{k} \cdot (\omega^\nu_k \otimes \omega^\nu_{n+1-k}); \mathsf{m}^\nu
\end{align*} 
$(iv)$: Again, we prove this by induction on $n$. For the base case $n=0$, we have that: 
\begin{align*}
\mathsf{T}(\omega^\nu);\omega^\nu_0 &=~ \mathsf{T}(\omega^\nu);\mathsf{T}(\pi_0); \nu = \mathsf{T}(\omega^\nu_0); \nu =Ê\mathsf{T}^2(\pi_0); \mathsf{T}(\nu);\nu =Ê\mathsf{T}^2(\pi_0); \mu;\nu  = \mu;\mathsf{T}(\pi_0); \nu =Ê\mu;\omega^\nu_0
\end{align*}
Now suppose that $\mathsf{T}(\omega^\nu);\omega^\nu_k= \mu; \omega^\nu_k$ holds for $k \leq n$, we now show it holds for $n+1$: 
\begin{align*}
&\mathsf{T}(\omega^\nu);\omega^\nu_{n+1}=~ \mathsf{T}(\omega^\nu);\left( \sum\limits^{n}_{k=0} \binom{n}{k} \cdot   \mathsf{d}; (\omega^\nu_k \otimes \pi_{n-k+1}); \mathsf{m}^\nu \right) \\
&=~Ê\sum\limits^{n}_{k=0} \binom{n}{k} \cdot   \mathsf{T}(\omega^\nu);\mathsf{d}; (\omega^\nu_k \otimes \pi_{n-k+1}); \mathsf{m}^\nu \\
&=~Ê\sum\limits^{n}_{k=0} \binom{n}{k} \cdot  \mathsf{d};( \mathsf{T}(\omega^\nu) \otimes \omega^\nu); (\omega^\nu_k \otimes \pi_{n-k+1}); \mathsf{m}^\nu \tag{Naturality of $\mathsf{d}$} \\
&=~Ê\sum\limits^{n}_{k=0} \binom{n}{k} \cdot  \mathsf{d};( \mu \otimes \omega^\nu); (\omega^\nu_k \otimes \pi_{n-k+1}); \mathsf{m}^\nu \tag{Induction Hypothesis} \\
&=~Ê\sum\limits^{n}_{k=0} \binom{n}{k} \cdot  \mathsf{d};( \mu \otimes 1); (\omega^\nu_k \otimes \omega^\nu_{n-k+1}); \mathsf{m}^\nu \\
&=~Ê\sum\limits^{n}_{k=0} \binom{n}{k} \cdot  \mathsf{d};( \mu \otimes 1); (\omega^\nu_k \otimes \left(\sum\limits^{n-k}_{j=0} \binom{n-k}{j} \cdot \mathsf{d}; (\omega^\nu_j \otimes \pi_{n-k-j+1}); \mathsf{m}^\nu \right)); \mathsf{m}^\nu \\
&=~Ê\sum\limits^{n}_{k=0}\sum\limits^{n-k}_{j=0} \binom{n}{k} \binom{n-k}{j} \cdot   \mathsf{d};( \mu \otimes \mathsf{d}); (\omega^\nu_k \otimes \omega^\nu_j \otimes \pi_{n-k-j+1}); (1 \otimes \mathsf{m}^\nu); \mathsf{m}^\nu \\
&=~Ê\sum\limits^{n}_{k=0}\sum\limits^{n-k}_{j=0} \binom{n}{k} \binom{n-k}{j} \cdot    \mathsf{d};( \mu \otimes \mathsf{d}); (\omega^\nu_k \otimes \omega^\nu_j \otimes \pi_{n-k-j+1}); (\mathsf{m}^\nu \otimes 1); \mathsf{m}^\nu \tag{Assoc. of Mult.} \\
&=~ \sum\limits^{n}_{k=0} \sum\limits^{k}_{j=0}  \binom{n}{k} \binom{k}{j} \cdot  \mathsf{d};(\mu \otimes \mathsf{d});(\omega^\nu_j \otimes \omega^\nu_{k-j} \otimes \pi_{n-k+1}); (\mathsf{m}^\nu \otimes 1); \mathsf{m}^\nu \tag{Reindexing} \\
&=~ \sum\limits^{n}_{k=0} \binom{n}{k} \cdot   \mathsf{d};(\mu \otimes \mathsf{d});(\left(\sum\limits^{k}_{j=0} \binom{k}{j} \cdot (\omega^\nu_j \otimes \omega^\nu_{k-j}); \mathsf{m}^\nu \right) \otimes \pi_{n-k+1}); \mathsf{m}^\nu \\
&=~ \sum\limits^{n}_{k=0} \binom{n}{k} \cdot   \mathsf{d};(\mu \otimes \mathsf{d});(\mathsf{m} \otimes 1);(\omega^\nu_k \otimes \pi_{n-k+1}); \mathsf{m}^\nu \tag{Lemma \ref{omegalemma} (iii)} \\
&=~Ê \sum\limits^{n}_{k=0} \binom{n}{k} \cdot   \mu; \mathsf{d}; (\omega^\nu_k \otimes \pi_{n-k+1}); \mathsf{m}^\nu \tag{Chain rule \textbf{[d.4]}}\\
&=~Ê\mu;\left( \sum\limits^{n}_{j=0} \binom{n}{j} \cdot   \mathsf{d}; (\omega^\nu_j \otimes \pi_{n-j+1}); \mathsf{m}^\nu \right) \\
&=~ \mu; \omega^\nu_{n+1} 
\end{align*}
\end{proof} 

\begin{proposition}\label{omegatalg} For a $\mathsf{T}$-algebra $(A, \nu)$, the pair $(\prod \limits^\infty_{n=0} A, \omega^\nu)$ is a $\mathsf{T}$-algebra.
\end{proposition} 
\begin{proof}
This follows from Lemma \ref{omegalemma} (ii) and (iv) and the universality of the product. \end{proof} 

\noindent The induced multiplication and unit of $(\prod \limits^\infty_{n=0} A, \omega^\nu)$ are precisely the Hurwitz product and unit \cite{keigher1997ring}. 

\begin{lemma}\label{hurwitzmult} For a $\mathsf{T}$-algebra $(A, \nu)$, the following equalities hold: 
\begin{enumerate}[{\em (i)}]
\item $\mathsf{u}^{\omega^\nu}; \pi_n = \begin{cases} \mathsf{u}^\nu & \text{if } n=0 \\
0 & \text{o.w.} 
\end{cases}$
\item $\mathsf{m}^{\omega^\nu}; \pi_n = \sum\limits^{n}_{k=0} \binom{n}{k} \cdot (\pi_k \otimes \pi_{n-k}); \mathsf{m}^\nu$
\end{enumerate}
\end{lemma} 
\begin{proof} For the unit, by Lemma \ref{omegalemma} $(i)$, we have that:
\[\mathsf{u}^{\omega^\nu}; \pi_n =\mathsf{u}; \omega^\nu; \pi_n= \mathsf{u}; \omega^\nu_n = \begin{cases} \mathsf{u}^\nu & \text{if } n=0 \\
0 & \text{o.w.} 
\end{cases} \]
For the multiplication, by Lemma \ref{omegalemma} $(ii)$ and $(iii)$, we have that:
\begin{align*}
\mathsf{m}^{\omega^\nu}; \pi_n &=~Ê(\eta \otimes \eta);\mathsf{m};\omega^\nu; \pi_n \\
&=~ (\eta \otimes \eta);\mathsf{m}; \omega^\nu_n \\
&=~ (\eta \otimes \eta); \left( \sum\limits^{n}_{k=0} \binom{n}{k} \cdot (\omega^\nu_k \otimes \omega^\nu_{n-k}); \mathsf{m}^\nu \right) \tag{Lemma \ref{omegalemma} (iii)} \\
&=~Ê\sum\limits^{n}_{k=0} \binom{n}{k} \cdot (\eta \otimes \eta);(\omega^\nu_k \otimes \omega^\nu_{n-k}); \mathsf{m}^\nu \\
&=~Ê \sum\limits^{n}_{k=0} \binom{n}{k} \cdot (\pi_k \otimes \pi_{n-k}); \mathsf{m}^\nu \tag{Lemma \ref{omegalemma} (ii)}
\end{align*}
\end{proof} 

Define the functor $\mathsf{H}: \mathbb{X}^\mathsf{T} \to \mathbb{X}^\mathsf{T}$ on objects as $\mathsf{H}(A, \nu):= (\prod \limits^\infty_{n=0} A, \omega^\nu)$ and on maps as ${\mathsf{H}(f) := \prod \limits^\infty_{n=0} f}$. 

\begin{lemma} $\mathsf{H}: \mathbb{X}^\mathsf{T} \to \mathbb{X}^\mathsf{T}$ is a functor. 
\end{lemma}
\begin{proof} We need only check that for a $\mathsf{T}$-algebra morphism $f$, the map $\mathsf{H}(f)$ is a $\mathsf{T}$-algebra morphism. So let $f: (A, \nu) \to (B, \nu^\prime)$ be a $\mathsf{T}$-algebra morphism. We first show by induction on $n$ that $\mathsf{T}\left(\prod \limits^\infty_{n=0} f\right); \omega^{\nu^\prime}_n= \omega^\nu_n; f$. For the base case $n=0$, we have that: 
\begin{align*}
\mathsf{T}\left(\prod \limits^\infty_{n=0} f\right); \omega^{\nu^\prime}_0 =Ê\mathsf{T}\left(\prod \limits^\infty_{n=0} f\right); \mathsf{T}(\pi_0); \nu^\prime = \mathsf{T}(\pi_0); \mathsf{T}(f); \nu^\prime = \mathsf{T}(\pi_0); \nu; f = \omega^\nu_0; f
\end{align*}
Now suppose the equality holds for $k \leq n$, we now show it for $n+1$: 
\begin{align*}
\mathsf{T}\left(\prod \limits^\infty_{n=0} f\right); \omega^\nu_{n+1} &=~Ê\mathsf{T}\left(\prod \limits^\infty_{n=0} f\right); \left( \sum\limits^{n}_{k=0} \binom{n}{k} \cdot \mathsf{d}; (\omega^\nu_k \otimes \pi_{n-k+1}); \mathsf{m}^\nu \right) \\
&=~  \sum\limits^{n}_{k=0} \binom{n}{k} \cdot \mathsf{T}\left(\prod \limits^\infty_{n=0} f\right);\mathsf{d}; (\omega^\nu_k \otimes \pi_{n-k+1}); \mathsf{m}^{\nu^\prime}  \\
&=~ \sum\limits^{n}_{k=0} \binom{n}{k} \cdot \mathsf{d};(\mathsf{T}\left(\prod \limits^\infty_{n=0} f \right) \otimes \prod \limits^\infty_{n=0} f);(\omega^\nu_k \otimes \pi_{n-k+1}); \mathsf{m}^{\nu^\prime} \tag{Naturality of $\mathsf{d}$}  \\
&=~  \sum\limits^{n}_{k=0} \binom{n}{k} \cdot \mathsf{d};(\omega^\nu_k \otimes \pi_{n-k+1});(f \otimes f); \mathsf{m}^{\nu^\prime} \tag{Induction Hypothesis}  \\
&=~ \sum\limits^{n}_{k=0} \binom{n}{k} \cdot \mathsf{d};(\omega^\nu_k \otimes \pi_{n-k+1}); \mathsf{m}^\nu; f \tag{\ref{monmorph}} \\
&=~Ê\left( \sum\limits^{n}_{k=0} \binom{n}{k} \cdot \mathsf{d}; (\omega^\nu_k \otimes \pi_{n-k+1}); \mathsf{m}^\nu \right); f \\
&=~ \omega^\nu_{n+1}; f
\end{align*}
Then by the universal property of the product, $\mathsf{T}\left(\prod \limits^\infty_{n=0} f\right); \omega^\nu= \omega^\nu; \prod \limits^\infty_{n=0} f$. 

\end{proof} 

Before defining a comonad structure on $\mathsf{H}$, we define a $\mathsf{T}$-differential algebra structure on $\mathsf{H}(A,\nu)$. For any object $A$, define the map $\mathsf{h}: \prod \limits^\infty_{n=0} A \to \prod \limits^\infty_{n=0} A$ as $\mathsf{h} := \langle \pi_{n+1} \vert ~ n \in \mathbb{N} \rangle$, that is, the unique map which for each $n \in \mathbb{N}$ makes the following diagram commute:
\begin{equation}\label{}\begin{gathered} \xymatrixcolsep{5pc}\xymatrixrowsep{2pc}\xymatrix{\prod \limits^\infty_{n=0} A \ar[r]^-{\mathsf{h}} \ar[dr]_-{\pi_{n+1}} & \prod \limits^\infty_{n=0} A \ar[d]^-{\pi_n} \\
& A
  } \end{gathered}\end{equation}
  
 \begin{proposition}\label{Hdiff} For a $\mathsf{T}$-algebra $(A,\nu)$, $(\prod \limits^\infty_{n=0} A, \omega^\nu, \mathsf{h})$ is a $\mathsf{T}$-differential algebra. Furthermore, if $f: (A, \nu) \to (B, \nu^\prime)$ is a $\mathsf{T}$-algebra morphism, then $\prod \limits^\infty_{n=0} f: (\prod \limits^\infty_{n=0} A, \omega^\nu, \mathsf{h}) \to (\prod \limits^\infty_{n=0} B, \omega^{\nu^\prime}, \mathsf{h})$ is a $\mathsf{T}$-differential algebra morphism.
\end{proposition} 
\begin{proof} We first show that $\mathsf{h}$ is a $\mathsf{T}$-derivation. Notice that for each $n \in \mathbb{N}$, we have that: 
\begin{align*}
\mathsf{d}; (\omega^\nu \otimes \mathsf{h}); \mathsf{m}^{\omega^\nu}; \pi_n &=~Ê\mathsf{d}; (\omega^\nu \otimes \mathsf{h}); \left(\sum\limits^{n}_{k=0} \binom{n}{k} \cdot (\pi_k \otimes \pi_{n-k}); \mathsf{m}^\nu \right) \tag{Lemma \ref{hurwitzmult} (ii)} \\
&=~Ê\sum\limits^{n}_{k=0} \binom{n}{k} \cdot \mathsf{d}; (\omega^\nu \otimes \mathsf{h}); (\pi_k \otimes \pi_{n-k}); \mathsf{m}^\nu \\
&=~Ê\sum\limits^{n}_{k=0} \binom{n}{k} \cdot \mathsf{d}; (\omega^\nu_k \otimes \pi_{n-k+1}); \mathsf{m}^\nu\\
&=~ \omega^\nu_{n+1} \\
&=~ \omega^\nu; \pi_{n+1}\\
&=~ \omega^\nu; \mathsf{h}; \pi_{n}
\end{align*}
Then by the universal property of the product, $\mathsf{d}; (\omega^\nu \otimes \mathsf{h}); \mathsf{m}^{\omega^\nu}= \omega^\nu; \mathsf{h}$.

 Next we show that $\prod \limits^\infty_{n=0} f ; \mathsf{h}= \mathsf{h}; \prod \limits^\infty_{n=0} f$. Notice that for each $n \in \mathbb{N}$, we have that: 
\[\prod \limits^\infty_{n=0} f; \mathsf{h}; \pi_n = \prod \limits^\infty_{n=0} f; \pi_{n+1} = \pi_{n+1}; f = \mathsf{h}; \pi_n; f \]
Then by the universal property of the product, $\prod \limits^\infty_{n=0} f ; \mathsf{h}= \mathsf{h}; \prod \limits^\infty_{n=0} f$.

\end{proof}  

Now consider a special morphism for $\mathsf{T}$-differential algebras in the presence of countable products. Let $(A, \nu, \mathsf{D})$ be a $\mathsf{T}$-differential algebra, then define the map $\mathsf{D}^\diamondsuit: A \to \prod \limits^\infty_{n=0} A$ as $\mathsf{D}^\diamondsuit := \langle \mathsf{D}^n \vert ~ n \in \mathbb{N} \rangle$, that is, the unique map which for each $n \in \mathbb{N}$ makes the following diagram commute:
\begin{equation}\label{deltadef}\begin{gathered} \xymatrixcolsep{5pc}\xymatrixrowsep{2pc}\xymatrix{A \ar[r]^-{\mathsf{D}^\diamondsuit} \ar[dr]_-{\mathsf{D}^n} & \prod \limits^\infty_{n=0} A \ar[d]^-{\pi_n} \\
& A
  } \end{gathered}\end{equation}
  and where recall that by convention $\mathsf{D}^0=1_A$.
 
\begin{lemma}\label{diamondlem} If $(A, \nu, \mathsf{D})$ is a $\mathsf{T}$-differential algebra, then $\mathsf{D}^\diamondsuit: (A, \nu, \mathsf{D}) \to (\prod \limits^\infty_{n=0} A, \omega, \mathsf{h})$ is a $\mathsf{T}$-differential morphism. Furthermore, if $f: (A, \nu, \mathsf{D}) \to (B, \nu^\prime, \mathsf{D}^\prime)$ is an $\mathsf{T}$-differential algebra morphism, then the following diagram commutes:
\begin{equation}\label{diamondH}\begin{gathered} \xymatrixcolsep{5pc}\xymatrix{ A \ar[d]_-{f} \ar[r]^-{\mathsf{D}^\diamondsuit} & \prod \limits^\infty_{n=0} A \ar[d]^-{\mathsf{H}(f)} \\
B \ar[r]_-{(\mathsf{D}^\prime)^\diamondsuit} & \prod\limits^\infty_{n=0} B 
  } \end{gathered}\end{equation}
\end{lemma}  
\begin{proof} We first show that $\mathsf{D}^\diamondsuit$ is a $\mathsf{T}$-algebra morphism by induction on $n$. For the base case $n=0$, we have that: 
\begin{align*}
\mathsf{T}(\mathsf{h}^\diamondsuit); \omega^\nu_0 = \mathsf{T}(\mathsf{h}^\diamondsuit); \mathsf{T}(\pi_0); \nu = \nu = \nu; \mathsf{h}^\diamondsuit; \pi_0
\end{align*}
Now suppose the equality holds for $k \leq n$, we now show it for $n+1$ by using the Fa\`a di Bruno formula: 
\begin{align*}
\mathsf{T}(\mathsf{D}^\diamondsuit); \omega^\nu_{n+1} &=~ \mathsf{T}(\mathsf{D}^\diamondsuit); \left(\sum\limits^{n}_{k=0} \binom{n}{k} \cdot \mathsf{d}; (\omega^\nu_k \otimes \pi_{n-k+1}); \mathsf{m}^\nu \right) \\
&=~ \sum\limits^{n}_{k=0} \binom{n}{k} \cdot \mathsf{T}(\mathsf{D}^\diamondsuit);\mathsf{d}; (\omega^\nu_k \otimes \pi_{n-k+1}); \mathsf{m}^\nu \\
&=~ \sum\limits^{n}_{k=0} \binom{n}{k} \cdot \mathsf{d};(\mathsf{T}(\mathsf{D}^\diamondsuit) \otimes \mathsf{D}^\diamondsuit);(\omega^\nu_k \otimes \pi_{n-k+1}); \mathsf{m}^\nu \tag{Naturality of $\mathsf{d}$} \\
&=~ \sum\limits^{n}_{k=0} \binom{n}{k} \cdot \mathsf{d};(\nu \otimes 1);(\mathsf{D}^\diamondsuit \otimes \mathsf{D}^\diamondsuit);(\pi_k \otimes \pi_{n-k+1}); \mathsf{m}^\nu \tag{Induction Hypothesis} \\
&=~  \sum\limits^{n}_{k=0} \binom{n}{k} \cdot \mathsf{d}; (\nu \otimes 1); (\mathsf{D}^k \otimes \mathsf{D}^{n-k+1}); \mathsf{m}^\nu \\
&=~ \nu; \mathsf{D}^{n+1} \tag{Fa\`a di Bruno Formula (\ref{der4})} \\
&=~ \nu; \mathsf{D}^\diamondsuit; \pi_{n+1}
\end{align*}
Then by the universal property of the product, $\mathsf{T}(\mathsf{D}^\diamondsuit); \omega^\nu = \nu;\mathsf{D}^\diamondsuit$. We must also show that $\mathsf{D}^\diamondsuit$ also preserves the $\mathsf{T}$-derivations. For each $n \in \mathbb{N}$, we have the following equality: 
\begin{align*}
\mathsf{D};\mathsf{D}^\diamondsuit; \pi_n = \mathsf{D}; \mathsf{D}^{n} = \mathsf{D}^{n+1} = \mathsf{D}^\diamondsuit; \pi_{n+1} = \mathsf{D}^\diamondsuit; \mathsf{h}; \pi_n
\end{align*}
Then by the universality of the product, $\mathsf{D};\mathsf{D}^\diamondsuit = \mathsf{D}^\diamondsuit; \mathsf{h}$. Therefore $\mathsf{D}$ is a $\mathsf{T}$-differential algebra morphism. 

Now suppose that $f$ is a $\mathsf{T}$-differential algebra morphism. Then for each $ n \in \mathbb{N}$, we have that: 
\[f; (\mathsf{D}^\prime)^\diamondsuit; \pi_n = f; (\mathsf{D}^\prime)^n = \mathsf{D}^n; f = \mathsf{D}^\diamondsuit; \pi_n; f = \mathsf{D}^\diamondsuit; \prod \limits^\infty_{n=0} f; \pi_n = \mathsf{D}^\diamondsuit; \mathsf{H}(f);\pi_n\]
Then by universality of the product, $f; (\mathsf{D}^\prime)^\diamondsuit= \mathsf{D}^\diamondsuit; \mathsf{H}(f)$. 
\end{proof} 

Consider now the projection $\pi_0: \prod \limits^\infty_{n=0}A \to A$ as a natural transformation $\pi_0: \mathsf{H}(A,\nu) \to (A,\nu)$ and the natural transformation $\mathsf{h}^\diamondsuit: \mathsf{H}(A,\nu) \to \mathsf{H}\mathsf{H}(A,\nu)$ defined as in (\ref{deltadef}) which is well defined since $\mathsf{h}$ is a $\mathsf{T}$-derivation (Proposition \ref{Hdiff}). 
  
\begin{proposition} $(\mathsf{H}, \mathsf{h}^\diamondsuit, \pi_0)$ is a comonad. 
\end{proposition} 
\begin{proof} We first explain why $(\mathsf{H}, \mathsf{h}^\diamondsuit, \pi_0)$ is well-defined. That $\mathsf{h}^\diamondsuit$ is $\mathsf{T}$-algebra morphism was shown in Lemma \ref{diamondlem}, while by definition of $\omega^\nu_0$, it follows that $\pi_0$ is also a $\mathsf{T}$-algebra morphism:
\[\omega^\nu;\pi_0 = \omega^\nu_0 = \mathsf{T}(\pi_0);\nu\]
Naturality of $\pi_0$ is obvious. While naturality of $\mathsf{h}^\diamondsuit$ follows from (\ref{diamondH}) since $\mathsf{H}(f)$ is a $\mathsf{T}$-differential algebra morphism (Proposition \ref{Hdiff}). It remains to show that the comonad identities, the duals of (\ref{monadeq}), are satisfied. Luckily, most of the work has been done. 
\begin{enumerate}
\item $\mathsf{h}^\diamondsuit; \pi_0=1$: This follows directly by definition of $\mathsf{h}^\diamondsuit$ (\ref{deltadef}). 
\item $\mathsf{h}^\diamondsuit; \mathsf{H}(\pi_0)=1$: For each $n \in \mathbb{N}$ we have that: 
\begin{align*}
\mathsf{h}^\diamondsuit; \mathsf{H}(\pi_0);\pi_n= \mathsf{h}^\diamondsuit;\prod \limits^\infty_{n=0} \pi_0;\pi_n &=~ \mathsf{h}^\diamondsuit; \pi_n; \pi_0 = \mathsf{h}^n; \pi_0 = \pi_n
\end{align*}
Then by the universal property of the product, $\mathsf{h}^\diamondsuit; \mathsf{H}(\pi_0) = 1$. 
\item $\mathsf{h}^\diamondsuit; \mathsf{H}(\mathsf{h}^\diamondsuit)= \mathsf{h}^\diamondsuit; \mathsf{h}^\diamondsuit$: This follows from directly from (\ref{diamondH}). 
\end{enumerate}
\end{proof} 

Similar to what was done in Section \ref{freesec}, we wish to show that the co-Eilenberg-Moore category of $\mathsf{H}$ is equivalent to category of $\mathsf{T}$-differential algebras. To do so, we will show that $\mathsf{H}$-coalgebras are precisely the $\mathsf{T}$-differential algebras of our codifferential category and that $\mathsf{H}$-coalgebra morphisms are precisely $\mathsf{T}$-differential algebra morphisms. Explicitly, an $\mathsf{H}$-coalgebra can be seen as a triple $(A, \nu, \kappa)$ consisting of a $\mathsf{T}$-algebra $(A, \nu)$ and a $\mathsf{T}$-algebra morphism $\kappa: (A, \nu) \to \mathsf{H}(A, \nu)$, such that $\kappa$ satisfies the dual of (\ref{Talgdef}), that is, $\kappa; \pi_0 =1$ and $\kappa; \mathsf{H}(\kappa) = \kappa; \mathsf{h}^\diamondsuit$. 

\begin{lemma}\label{HalgTalg} Let $(A, \nu, \kappa)$ be a $\mathsf{H}$-coalgebra. Define the map $\mathsf{D}^\kappa: A \to A$ as follows:
\begin{equation}\label{}\begin{gathered} \xymatrixcolsep{5pc}\xymatrix{ A \ar[r]^-{\kappa} & \prod \limits^\infty_{n=0} A \ar[r]^-{\pi_1} & A
  } \end{gathered}\end{equation}
Then $(A, \nu, \mathsf{D}^\kappa)$ is a $\mathsf{T}$-differential algebra. Furthermore, if $f: (A, \nu, \kappa) \to (B, \nu^\prime, \kappa^\prime)$ is an $\mathsf{H}$-coalgebra morphism, then $f: (A, \nu, \mathsf{D}^\kappa) \to (B, \nu^\prime, \mathsf{D}^{\kappa^\prime})$ is a $\mathsf{T}$-differential algebra morphism. 
\end{lemma}
\begin{proof} That $\mathsf{D}^\kappa$ is a $\mathsf{T}$-derivation follows from the fact that $\kappa$ is a $\mathsf{T}$-algebra morphism. 
\begin{align*}
\nu; \mathsf{D}^\kappa &=~ \nu; \kappa; \pi_1 \\
&=~Ê\mathsf{T}(\kappa); \omega; \pi_1 \tag{\ref{Talgdef}} \\
&=~Ê\mathsf{T}(\kappa); \omega_1\\
&=~ \mathsf{T}(\kappa); \mathsf{d}; (\omega^\nu_0 \otimes \pi_{1}); \mathsf{m}^\nu \\
&=~ \mathsf{T}(\kappa); \mathsf{d}; (\mathsf{T}(\pi_0) \otimes \pi_{1}); (\nu \otimes 1); \mathsf{m}^\nu \\
&=~ \mathsf{d};(\mathsf{T}(\kappa) \otimes \kappa); (\mathsf{T}(\pi_0) \otimes \pi_{1}); (\nu \otimes 1); \mathsf{m}^\nu \tag{Naturality of $\mathsf{d}$} \\
&=~Ê\mathsf{d}; (\nu \otimes \mathsf{D}^\kappa); \mathsf{m}^\nu \tag{\ref{Talgdef}} 
\end{align*}
Now suppose that $f$ is a $\mathsf{H}$-coalgebra morphism and therefore a $\mathsf{T}$-algebra by definition. It remains to show that that $f$ also commutes with the $\mathsf{T}$-derivations. 
\begin{align*}
f; \mathsf{D}^{\kappa^\prime} = f; \kappa^\prime; \pi_1 = \kappa; \mathsf{H}(f); \pi_1 = \kappa; \prod \limits^\infty_{n=0} f; \pi_1 = \kappa; \pi_1; f = \mathsf{D}^\kappa; f
\end{align*}
\end{proof} 

The converse of Lemma \ref{HalgTalg} follows mostly from Lemma \ref{diamondlem}: 

\begin{lemma}\label{TalgHalg} Let $(A, \nu, \mathsf{D})$ be a $\mathsf{T}$-differential algebra, then $(A, \nu, \mathsf{D}^\diamondsuit)$ is an $\mathsf{H}$-coalgebra. Furthermore, if $f: (A, \nu, \mathsf{D}) \to (B, \nu^\prime, \mathsf{D}^\prime)$ is an $\mathsf{T}$-differential algebra morphism, then $f: (A, \nu, \mathsf{D}^\diamondsuit) \to (B, \nu^\prime, (\mathsf{D}^\prime)^\diamondsuit)$ is a $\mathsf{H}$-coalgebra morphism. 
\end{lemma}
\begin{proof} We have already shown that $\mathsf{D}^\diamondsuit$ is a $\mathsf{T}$-algebra morphism in Lemma \ref{diamondlem}. Therefore it remains to check that $\mathsf{D}^\diamondsuit$ satisfies the $\mathsf{H}$-coalgebra identities. That $\mathsf{D}^\diamondsuit; \pi_0 = 1$ follows by definition of $\mathsf{D}^\diamondsuit$ (\ref{deltadef}). For the other $\mathsf{H}$-coalgebra identity, since $\mathsf{D}^\diamondsuit: (A, \nu, \mathsf{D}) \to (\prod \limits^\infty_{n=0} A, \omega, \mathsf{h})$ is a $\mathsf{T}$-differential algebra morphism (Lemma \ref{diamondlem}), then $\mathsf{D}^\diamondsuit$ satisfies (\ref{diamondH}), which is precisely that $\mathsf{D}^\diamondsuit; \mathsf{h}^\diamondsuit = \mathsf{D}^\diamondsuit; \mathsf{H}(\mathsf{D}^\diamondsuit)$. Similarly, if $f$ is a $\mathsf{T}$-differential algebra morphism, then $f$ also satisfies (\ref{diamondH}), which is precisely the requirement that $f$ be an $\mathsf{H}$-coalgebra morphism.
\end{proof} 

\begin{theorem}\label{cofreethm} The co-Eilenberg-Moore category of the comonad $(\mathsf{H}, \mathsf{h}^\diamondsuit, \pi_0)$ is isomorphic to category of $\mathsf{T}$-differential algebras of $\mathbb{X}$. \end{theorem} 
\begin{proof} As in the proof of Theorem \ref{freethm}, it suffices to show that the constructions of Lemma \ref{HalgTalg} and Lemma \ref{TalgHalg} are inverses of each other. 

Starting with a $\mathsf{H}$-coalgebra $(A, \nu, \kappa)$, we must show that $(\mathsf{D}^\kappa)^\diamondsuit = \kappa$. We will do so by showing by induction that $(\mathsf{D}^\kappa)^\diamondsuit;\pi_n = \kappa; \pi_n$. For the base case $n=0$, since $\kappa$ is an $\mathsf{H}$-coalgebra structure, we have that: 
\begin{align*}
(\mathsf{D}^\kappa)^\diamondsuit; \pi_0 = 1 = \kappa; \pi_0 
\end{align*}
Now suppose that the desired identity holds for $k\leq n$, we now show it for $n+1$. Here we use that since $\kappa$ is an $\mathsf{H}$-coalgebra morphism, then by Lemma \ref{HalgTalg} $\kappa: (A, \nu, \mathsf{D}^\diamondsuit) \to (\prod \limits^\infty_{n=0} A, \omega, \mathsf{h})$ is a $\mathsf{T}$-differential algebra morphism. 
\begin{align*}
(\mathsf{D}^\kappa)^\diamondsuit; \pi_{n+1} = (\mathsf{D}^\kappa)^{n+1} = \mathsf{D}^\kappa; (\mathsf{D}^\kappa)^{n} = \mathsf{D}^\kappa; \kappa; \pi_n = \kappa; \mathsf{h}; \pi_n = \kappa; \pi_{n+1}
\end{align*}
Conversly, starting with a $\mathsf{T}$-differential algebra $(A, \nu, \mathsf{D})$, we must show that $\mathsf{D}^{\mathsf{D}^\diamondsuit} = \mathsf{D}$. This however follows by definition (\ref{deltadef}): 
\[\mathsf{D}^{\mathsf{D}^\diamondsuit} = \mathsf{D}^\diamondsuit; \pi_1= \mathsf{D}\]
\end{proof} 

As a consequence of Theorem \ref{cofreethm}, $(\prod \limits^\infty_{n=0} A, \omega^\nu, \mathsf{h})$ is the cofree $\mathsf{T}$-differential algebra over a $\mathsf{T}$-algebra $(A,\nu)$.

We now turn our attention to power series over $\mathsf{T}$-algebras in codifferential categories. These power series over $\mathsf{T}$-algebras will of course be generalization of classical power series rings over algebras. The construction of power series over $\mathsf{T}$-algebras is similar to that of cofree $\mathsf{T}$-differential algebras (therefore we will omit most proofs) in the same way that power series rings are very similar to Hurwitz series rings. In fact, in Proposition \ref{powerH} we will show that (in the presence of positive rationals) power series over $\mathsf{T}$-algebras are isomorphic to cofree $\mathsf{T}$-differential algebras.

Classically, for an $R$-algebra $A$, the underlying set of the power series ring over $A$ is isomorphic to the set $\prod \limits^\infty_{n=0} A$ (just like the Hurwitz series ring). As before, we will need to define a $\mathsf{T}$-differential algebra structure on $\prod \limits^\infty_{n=0} A$ such that the underlying differential algebra structure is precisely that of power series rings. 

For a $\mathsf{T}$-algebra $(A, \nu)$, define for each $n \in \mathbb{N}$ the maps $\delta_n: \mathsf{T}\left( \prod \limits^\infty_{n=0} A \right) \to A$ inductively as: 
\begin{enumerate}[{\em (i)}]
\item $\delta_0 := \mathsf{T}(\pi_0); \nu$
\item $\delta_{n+1} := \sum\limits^{n}_{k=0} \mathsf{d}; (\delta_k \otimes \pi_{n-k+1}); \mathsf{m}^\nu$
\end{enumerate}
Notice that $\delta_n$ is very similar to $\omega_n$ but without scalar multiplying by the binomial coefficients. Define the map $\delta: \mathsf{T}\left( \prod \limits^\infty_{n=0} A \right) \to \prod \limits^\infty_{n=0} A$ as $\delta := \langle \delta_n \vert ~ n \in \mathbb{N} \rangle$, that is, the unique map such that for each $n \in \mathbb{N}$, the following diagram commutes: 
\begin{equation}\label{omegadef}\begin{gathered} \xymatrixcolsep{5pc}\xymatrix{\prod \limits^\infty_{n=0} A \ar[r]^-{\delta} \ar[dr]_-{\delta_{n}} & \prod \limits^\infty_{n=0} A \ar[d]^-{\pi_n} \\ 
& A
  } \end{gathered}\end{equation}
  
\begin{lemma}\label{deltalemma} For a $\mathsf{T}$-algebra $(A, \nu)$, the following equalities hold (for all $n \in \mathbb{N}$): 
\begin{enumerate}[{\em (i)}]
\item $\mathsf{u} ; \delta_n = \begin{cases} \mathsf{u}^\nu & \text{if } n=0 \\
0 & \text{o.w.} 
\end{cases}$
\item $\eta; \delta_n = \pi_n$
\item $\mathsf{m}; \delta_n = \sum\limits^{n}_{k=0} (\delta_k \otimes \delta_{n-k}); \mathsf{m}^\nu$
\item $\mu; \delta_n = \mathsf{T}(\delta); \delta_n$
\end{enumerate}
\end{lemma}

\begin{proposition}\label{deltatalg} For a $\mathsf{T}$-algebra $(A, \nu)$, the pair $(\prod \limits^\infty_{n=0} A, \delta)$ is a $\mathsf{T}$-algebra.
\end{proposition} 

The induced multiplication and unit of $(\prod \limits^\infty_{n=0} A, \delta)$ are precisely the multiplication and unit of classical power series rings: 

\begin{lemma}\label{psmult} For a $\mathsf{T}$-algebra $(A, \nu)$, the following equalities hold: 
\begin{enumerate}[{\em (i)}]
\item $\mathsf{u}^\delta; \pi_n = \begin{cases} \mathsf{u}^\nu & \text{if } n=0 \\
0 & \text{o.w.} 
\end{cases}$
\item $\mathsf{m}^\delta; \pi_n = \sum\limits^{n}_{k=0} (\pi_k \otimes \pi_{n-k}); \mathsf{m}^\nu$
\end{enumerate}
\end{lemma} 
Notice again how $\mathsf{m}^\delta$ and $\mathsf{u}^\delta$ are similar to $\mathsf{m}^\omega$ and $\mathsf{u}^\omega$ (Lemma \ref{hurwitzmult}) without scalar multiplying by the binomial coefficients. 

Define the map $\mathsf{p}: \prod \limits^\infty_{n=0} A \to \prod \limits^\infty_{n=0} A$ as $\mathsf{p} := \langle (n+1) \cdot \pi_{n+1} \vert ~ n \in \mathbb{N} \rangle$, that is, the unique map which for each $n \in \mathbb{N}$ makes the following diagram commute:
\begin{equation}\label{}\begin{gathered} \xymatrixcolsep{5pc}\xymatrixrowsep{4pc}\xymatrix{\prod \limits^\infty_{n=0} A \ar[r]^-{\mathsf{p}} \ar[dr]_-{(n+1) \cdot \pi_{n+1}} & \prod \limits^\infty_{n=0} A \ar[d]^-{\pi_n} \\
& A
  } \end{gathered}\end{equation}
  
\begin{proposition} The triple $(\prod \limits^\infty_{n=0} A, \delta, \mathsf{p})$ is a $\mathsf{T}$-differential algebra. And therefore, $(\prod \limits^\infty_{n=0} A, \delta, \mathsf{p}^\diamondsuit)$ is a $\mathsf{H}$-coalgebra. 
\end{proposition} 

The $\mathsf{T}$-derivation $\mathsf{p}$ captures the standard differentiation of power series (\ref{classicalpower}), where in particular $\mathsf{D}(x^{n+1})= (n+1)x^{n}$. 

As mentioned in Example \ref{cofreediff}, power series rings and Hurwitz series rings are closely related to one another \cite{keigher1997ring}. This relation still holds true between cofree $\mathsf{T}$-differential algebras and power series over $\mathsf{T}$-algebras.

\begin{proposition}\label{powerH} Consider a codifferential category with countable products and which admits non-negatives rationals $\mathbb{Q}_{\geq 0}$, that is, is enriched over $\mathbb{Q}_{\geq 0}$-modules. Then for every $\mathsf{T}$-algebra $(A, \nu)$, $(\prod \limits^\infty_{n=0} A, \delta, \mathsf{p})$ and $(\prod \limits^\infty_{n=0} A, \omega, \mathsf{h})$ are isomorphic as $\mathsf{T}$-differential algebras. \end{proposition} 
\begin{proof} Define the $\mathsf{T}$-differential morphism $\psi: (\prod \limits^\infty_{n=0} A, \delta, \mathsf{p}) \to (\prod \limits^\infty_{n=0} A, \omega, \mathsf{h})$ as follows: 
\begin{equation}\label{psidef}\begin{gathered} \psi := \xymatrixcolsep{5pc}\xymatrix{\prod \limits^\infty_{n=0} A \ar[r]^-{\mathsf{p}^\diamondsuit} & \prod \limits^\infty_{m=0} \prod \limits^\infty_{n=0} A \ar[r]^-{\prod \limits^\infty_{m=0} \pi_0} & \prod \limits^\infty_{n=0} A 
  } \end{gathered}\end{equation}
  which is a $\mathsf{T}$-differential morphisms by construction (Proposition \ref{Hdiff} and Lemma \ref{diamondlem}). Note that this map can be constructed in arbitrary codifferential categories with countable products. 
  
We first note that $\psi$ admits a nice alternative description. Indeed, for each $n \in \mathbb{N}$, we get: 
  \begin{align*}
\psi; \pi_n = \mathsf{p}^\diamondsuit; \prod \limits^\infty_{n=0} \pi_0; \pi_n = \mathsf{p}^\diamondsuit; \pi_n; \pi_0 = \mathsf{p}^n ; \pi_0 = n\oc \cdot \pi_n
\end{align*}
Then by the universal property of the product, $\psi= \langle n\oc \cdot \pi_{n} \rangle$. And since we are working in a codifferential category with positive rationals, one defines the inverse of $\psi$ as $\psi^{-1} := \langle \frac{1}{n\oc} \cdot \pi_{n} \vert~ n \in \mathbb{N} \rangle$. 

\end{proof} 

\section{Conclusion and Future Directions}\label{consec}

In this paper we introduced and studied the notion of $\mathsf{T}$-differential algebras which provide a generalization of differential algebras to the context of differential categories. As such, we have been able to add the theory of differential algebras nicely into the theory of differential categories. There are many interesting potential directions for future work with $\mathsf{T}$-differential algebras. One could study $\mathsf{T}$-differential algebras in more exotic differential categories such as the category of convenient vector spaces \cite{blute2010convenient} or other models of differential linear logic. Or one could study linear differential equations for $\mathsf{T}$-differential algebras and find solutions to linear combinations of $\mathsf{D}^n$. One could also ask whether it is possible to generalize the notion of weighted differential algebras to the context of differential categories. The first question to solve when generalizing weighted differential algebras is whether these should be defined for regular differential categories or a weighted version of a differential category. One of the difficulties with defining the latter is that there is not a general weighted chain rule formula. On the other hand, there has recently been much effort put into the study of the integration and the fundamental theorem of calculus in the context differential categories \cite{bagnol2016shuffle, cockett_lemay_2018, delaney2018generalized, ehrhard2017introduction}. Hopefully one should be able to develop $\mathsf{T}$-Rota-Baxter algebras for differential categories and construct free $\mathsf{T}$-Rota-Baxter algebras. One could then study $\mathsf{T}$-differential algebras with compatible $\mathsf{T}$-Rota-Baxter algebra structure and the fundamental theorems of calculus. 

\bibliographystyle{spmpsci}      % mathematics and physical sciences
\bibliography{diffalg}   % name your BibTeX data base

\end{document}